\documentclass[12pt,reqno]{amsart}
\usepackage{amsmath, amsfonts,epsfig, amssymb,amsbsy,eucal,mathrsfs,float, amsthm,mathtools,amsrefs}

\usepackage[normalem]{ulem}

\usepackage{graphicx,xcolor}
\usepackage{dynkin-diagrams}
\usepackage{fullpage}

\usepackage{tikz}

\usepackage[all]{xy}

\usepackage[enableskew]{youngtab}

\usepackage{tikz-cd}
\usepackage{tikz}

\usepackage{hyperref}
\hypersetup{
colorlinks = true,
urlcolor = blue,
citecolor = .
}

\usepackage{todonotes}

\numberwithin{equation}{section}
\theoremstyle{plain}

\newtheorem{theorem}{Theorem}[section]
\newtheorem{lemma}[theorem]{Lemma}
\newtheorem{proposition}[theorem]{Proposition}
\newtheorem{corollary}[theorem]{Corollary}

\theoremstyle{definition}

\newtheorem{example}[theorem]{Example}
\newtheorem{observation}[theorem]{Observation}

\theoremstyle{remark}

\newtheorem{remark}[theorem]{Remark}

\DeclareMathOperator{\OG}{OG}

\DeclareMathOperator{\Ext}{Ext}
\DeclareMathOperator{\Hom}{Hom}

\DeclareMathOperator{\Aut}{Aut}

\DeclareMathOperator{\Rep}{Rep}
\DeclareMathOperator{\Coh}{Coh}

\DeclareMathOperator{\SO}{SO}

\DeclareMathOperator{\Res}{Res}
\DeclareMathOperator{\Diff}{Diff}
\DeclareMathOperator{\Lie}{Lie}
\DeclareMathOperator{\centralizer}{C}

\DeclareMathOperator{\rk}{rk}

\renewcommand{\ss}{\mathbf{ss}}
\newcommand{\bfP}{\mathbf{P}}

\newcommand{\bC}{{\mathbb C}}
\newcommand{\bG}{{\mathbb G}}

\newcommand{\bP}{{\mathbb P}}
\newcommand{\bQ}{{\mathbb Q}}
\newcommand{\bZ}{{\mathbb Z}}

\newcommand{\rA}{\mathrm{A}}
\newcommand{\rB}{\mathrm{B}}
\newcommand{\rC}{\mathrm{C}}
\newcommand{\rD}{\mathrm{D}}
\newcommand{\rE}{\mathrm{E}}
\newcommand{\rF}{\mathrm{F}}
\newcommand{\rG}{\mathrm{G}}

\newcommand{\rL}{\mathrm{L}}
\newcommand{\rLprime}{\mathrm{L'}}
\newcommand{\rP}{\mathrm{P}}
\newcommand{\rT}{\mathrm{T}}

\newcommand{\rV}{\mathrm{V}}
\newcommand{\rW}{\mathrm{W}}

\newcommand{\cA}{\mathcal{A}}

\newcommand{\cC}{\mathcal{C}}
\newcommand{\cD}{\mathcal{D}}

\newcommand{\cO}{\mathcal{O}}
\newcommand{\cR}{\mathcal{R}}

\newcommand{\cU}{\mathcal{U}}

\newcommand{\Db}{{\mathbf D^{\mathrm{b}}}}

\newcommand{\baromega}{\bar{\omega}}
\newcommand{\barcU}{\bar{\cU}}

\newcommand{\tildeomega}{\widetilde{\omega}}

\begin{document}

\title{On the derived category of the adjoint \\ Grassmannian of type F}

\author{Maxim N. Smirnov}
\address{
\parbox{0.95\textwidth}{
Universit\"at Augsburg,
Institut f\"ur Mathematik,
Universit\"atsstr.~14,
86159 Augsburg,
Germany
\smallskip
}}
\email{maxim.smirnov@math.uni-augsburg.de}
\email{maxim.n.smirnov@gmail.com}

\thanks{M.S. was partially supported by the Deutsche Forschungsgemeinschaft (DFG, German Research Foundation) --- Projektnummer 448537907.}

\maketitle

% Draft, \today

\begin{center}
  \textit{In blessed memory of my teacher Yuri Ivanovich Manin}
\end{center}

\vspace{1\baselineskip}

\begin{flushright}
\emph{``For us, there is only the trying. The rest is not our business.''}

\vspace{0.5\baselineskip}

--- T.\ S.\ Eliot\\
\emph{East Coker}, \emph{Four Quartets}
\end{flushright}

\vspace{0.5\baselineskip}

\begin{abstract}
We construct a full rectangular Lefschetz collection in the derived category of
the adjoint Grassmannian in type $\rF_4$. This gives the first example of a full
exceptional collection on this variety and also completes the proof of a conjecture due to Alexander Kuznetsov
and the author that relates the structure of the derived category of coherent sheaves
to the small quantum cohomology in the case of adjoint varieties for non-simply laced groups.
\end{abstract}

% \tableofcontents

\section{Introduction}
\label{section:introduction}

Exceptional collections in derived categories of rational homogeneous spaces $\rG/\rP$ have a long history going back to the pioneering works of Beilinson \cite{Be} and Kapranov \cite{Ka}. It is conjectured that the derived category $\Db(\rG/\rP)$ always has a full exceptional collection and many cases of the conjecture are known by now. For an overview of the state of
the art from a few years ago we refer to~\cite[\S1.1]{KuPo}. The main advances since then are \cite{Fo19,Gu18,KS21, BKS}, where exceptional collections were constructed on some homogeneous varieties of symplectic and orthogonal groups, and the coadjoint variety of type $\rF_4$.

For homogeneous spaces of classical types $\rA_n$, $\rB_n$, $\rC_n$, $\rD_n$ and in the exceptional type
of small rank $\rG_2$ the story is relatively well studied (see~\cite{KuPo} and~\cite[\S6.4]{K06}). On the contrary,
for the exceptional types $\rF_4$, $\rE_6$, $\rE_7$, $\rE_8$ there were only two known instances of the
conjecture: the cominuscule variety of type $\rE_6$ (also called the \textsf{Cayley plane}) \cite{FM, Manivel} and
the coadjoint variety of type $\rF_4$ (this is a general hyperplane section of the Cayley plane) \cite{BKS}.

In this paper we deal with one more case in type $\rF_4$ by constructing a full exceptional Lefschetz collection (see \cite{K07}) in the bounded derived category of coherent sheaves on the adjoint variety of type $\rF_4$. To state our result we introduce a few facts on the geometry of this variety.

Let $\rG$ be a simple algebraic group of Dynkin type $\rF_4$ and let $\rP = \rP_1 \subset \rG$ be the maximal parabolic subgroup associated with the first vertex of its Dynkin diagram
\begin{equation}\label{diagram:F4}
  \dynkin[edge length = 2em,labels*={1,...,4}, arrow width=2mm, scale=1.7]F{*ooo}
\end{equation}
The homogeneous variety $X = \rG/\rP$ has a natural $\rG$-equivariant projective embedding
\begin{equation}\label{eq:projective-embedding-lie-algebra}
  X \subset \bP(\mathfrak{g}),
\end{equation}
where $\rG$ acts on its Lie algebra $\mathfrak{g}$ via the adjoint action.
In this embedding $X$ gets identified with the $\rG$-orbit of the highest root.
Hence, $X$ is called the \textsf{adjoint Grassmannian}
in type $\rF_4$. The variety $X$ is a smooth projective Fano variety of Picard rank $1$ with
\begin{equation}\label{eq:dimension-and-index}
  \dim X = 15 \quad \text{and} \quad \omega_X = \cO_X(-8),
\end{equation}
where $\cO_X(1)$ is the ample generator of the Picard group of $X$.

Let $\cU^{\omega_4}$ be the irreducible $\rG$-equivariant vector bundle on $X$ associated with the fourth vertex of the Dynkin diagram \eqref{diagram:F4} (see Section \ref{section:preliminaries} for more details). It has the following properties
\begin{equation*}
  \rk (\cU^{\omega_4}) = 6 \quad \text{and} \quad \det (\cU^{\omega_4}) = \cO(3).
\end{equation*}
This bundle turns out to be exceptional and is used, together with its twists, in our exceptional collection (see Theorem \ref{theorem:F4-P1-intro} below).

Let us now consider the tangent bundle $T_X$. It is a $\rG$-equivariant vector bundle on $X$. Note, however, that $T_X$ is not irreducible. More importantly, the vector bundle $T_X$ is not exceptional. However, a direct computation shows that we have
\begin{equation*}
  \Ext^\bullet(T_X, \cO_X) = \Bbbk[-1].
\end{equation*}
Thus, there exists a unique non-trivial $\rG$-equivariant extension
\begin{equation}\label{eq:exact-sequence-extension-tangent-bundle-introduction}
  0 \to \cO_X \to \widetilde{T}_X \to T_X \to 0.
\end{equation}
From \eqref{eq:dimension-and-index} and \eqref{eq:exact-sequence-extension-tangent-bundle-introduction}
we immediately obtain
\begin{equation*}
  \rk (\widetilde{T}_X) = 16 \quad \text{and} \quad \det (\widetilde{T}_X) = \cO_X(8).
\end{equation*}
The vector bundle $\widetilde{T}_X$ is exceptional and is used, together with its twists, in our exceptional collection (see Theorem \ref{theorem:F4-P1-intro} below). Let us also note here that $\widetilde{T}_X$ is isomorphic to the sheaf $\Diff_X^{\leq 1}$ of differential operators on $X$ of order less or equal to $1$.

\smallskip

Now we are ready to state our main result.
\begin{theorem}\label{theorem:F4-P1-intro}
  Let $X$ be the adjoint Grassmannian of Dynkin type $\rF_4$. There exists a full exceptional collection
  of length $24$
  \begin{equation}\label{eq:collection-F4-P1-intro}
    \Db(X) = \big\langle
    \cO_X, \cU^{\omega_4}, \widetilde{T}_X, \cO_X(1), \cU^{\omega_4}(1), \widetilde{T}_X(1), \dots,
    \cO_X(7), \cU^{\omega_4}(7), \widetilde{T}_X(7) \big\rangle
  \end{equation}
  consisting of $\rG$-equivariant vector bundles.
  This collection is an $\Aut(X)$-invariant rectangular Lefschetz collection.
\end{theorem}

The concept of adjoint varieties is meaningful in any Dynkin type and their
derived categories share some common features. In a joint work with Alexander Kuznetsov
\cite{KS21}, drawing motivation from the structure of the small quantum cohomology
of these varieties, we have proposed some conjectures on the structure of their
derived categories. In the particular case of adjoint varieties in non-simply
laced Dynkin types (i.e., types $\rB_n$, $\rC_n$, $\rF_4$, or $\rG_2$) our conjecture
\cite[Conjecture 1.7]{KS21} predicts the existence of a full $\Aut$-invariant
rectangular Lefschetz collection. Therefore, since \eqref{eq:collection-F4-P1-intro}
has all these properties, we obtain the following.
\begin{corollary}\label{corollary:F4-P1-intro}
  The conjecture \cite[Conjecture 1.7]{KS21} holds for the adjoint Grassmannian in type $\rF_4$.
\end{corollary}
Since in types $\rB_n$, $\rC_n$, and $\rG_2$ the conjecture \cite[Conjecture 1.7]{KS21}
is already known to hold, Corollary \ref{corollary:F4-P1-intro} finishes its proof
by treating the only remaining case.

\begin{remark}
By \cite[Proposition 6.3]{ChPe} the small quantum cohomology ring of $X$ is known to be semisimple. Hence, Theorem~\ref{theorem:F4-P1-intro} provides yet another instance where the first part of Dubrovin's conjecture (see~\cite{Du}) is known to hold.
\end{remark}

The proof of Theorem \ref{theorem:F4-P1-intro} naturally splits into two steps.
First, by a direct computation we show that the collection \eqref{eq:collection-F4-P1-intro}
is exceptional. Second, we prove that \eqref{eq:collection-F4-P1-intro} is full by
restricting it to a family of subvarieties of $X$ that sweep the whole $X$ and are
all isomorphic to the even orthogonal Grassmannian $\OG(2,8)$. It is essential that
for $\OG(2,8)$ a suitable exceptional collection is already known by \cite{KS21}.
In this respect the proof is similar to the proofs of fullness given in \cite{KS21, Ku08a}.
The extension $\widetilde{T}_{\OG(2,8)}$ of the tangent bundle $T_{\OG(2,8)}$ by the structure sheaf as in \eqref{eq:exact-sequence-extension-tangent-bundle-introduction} also appears in the aforementioned collection on $\OG(2,8)$.

\begin{observation}
  Let us discuss how the use of the extension $\widetilde{T}_X$ of the tangent
  bundle $T_X$ by the structure sheaf $\cO_X$ to construct exceptional collections
  can be extended to other homogenous spaces. We restrict our attention to homogenous
  spaces $\rG / \rP$ with $\rP$ maximal parabolic and we label them by pairs
  $(D, k)$, where $D$ is the Dynkin diagram of $\rG$ and $k$ is the vertex in $D$ that
  corresponds to the parabolic $\rP$.
  We use the Bourbaki numbering of vertices in Dynkin diagrams \cite{Bourbaki}.

  For any $\rG / \rP$  as above we have
  \begin{equation*}
    \Ext^\bullet(T_{\rG / \rP}, \cO_{\rG / \rP}) = H^\bullet(\rG / \rP, \Omega_{\rG / \rP}) = \Bbbk[-1].
  \end{equation*}
  Hence, there always exists a unique non-trivial $\rG$-equivariant extension $\widetilde{T}_{\rG / \rP}$ of
  the tangent bundle $T_{\rG / \rP}$ by the structure sheaf $\cO_{\rG / \rP}$.
  SageMath \cite{sagemath} computations suggest the following patterns:
  \begin{enumerate}
    \item For any $\rG / \rP$ we have
    \begin{equation}\label{eq:vanishing-higher-exts-intro}
      \Ext^i(T_{\rG / \rP}, T_{\rG / \rP}) = 0 \quad \text{for} \quad i \geq 2.
    \end{equation}

    \item \label{exceptions-1-intro}
    The extension $\widetilde{T}_{\rG / \rP}$ is an exceptional object for any
    $\rG / \rP$ with the following exceptions:
    \begin{equation}\label{eq:list-of-exceptions-1-intro}
      \begin{aligned}
        & (\rA_n, 1) \text{  and  } (\rA_n, n); \\
        & (\rB_n, n-1) \text{  for  } n \geq 3 \text{  and also  } (\rB_2, 2); \\
        & (\rC_n, 1) \text{  and  } (\rC_n, 2) \text{  for  } n \geq 3 \text{  and also  } (\rC_2, 1); \\
        & (\rF_4, 4) \text{  and  } (\rG_2, 2).
      \end{aligned}
    \end{equation}
    These exceptions seem to be closely related to the failure of the Yoneda map
    \begin{equation}\label{eq:yoneda-map-intro}
      \Ext^1(T_{\rG / \rP}, \cO_{\rG / \rP}) \times \Ext^0(\cO_{\rG / \rP}, T_{\rG / \rP})
      \to \Ext^1(T_{\rG / \rP}, T_{\rG / \rP})
    \end{equation}
    to be an isomorphism. In the cases $(\rA_n, 1)$, $(\rA_n, n)$, $(\rB_2, 2)$ and
    $(\rC_n, 1)$ the variety $\rG / \rP$ is a projective space and it is well-known
    that in this case already $T_{\rG / \rP}$ itself is an exceptional object. Hence,
    $\Ext^1(T_{\rG / \rP}, T_{\rG / \rP})$ vanishes and \eqref{eq:yoneda-map-intro}
    is not injective. For all other exceptions the map \eqref{eq:yoneda-map-intro}
    is not surjective.

    \medskip

    \item
    It is well-known (for example, see \cite[Section IV.8, Theorem 2]{Akhiezer}) that for any $\rG / \rP$
    except for $(\rB_n, n)$, $(\rC_n, 1)$ and $(\rG_2, 1)$ we have an isomorphism
    \begin{equation*}
      H^0(\rG / \rP, T_{\rG / \rP}) = \mathfrak{g}.
    \end{equation*}
    Hence, the above discussion
    on the Yoneda map \eqref{eq:yoneda-map-intro} suggests that for any $\rG / \rP$
    with the exceptions given by \eqref{eq:list-of-exceptions-1-intro},
    and also $(\rB_n, n)$ and $(\rG_2, 1)$,
    we have an isomorphism
    \begin{equation}\label{eq:moduli}
      \Ext^1(T_{\rG / \rP}, T_{\rG / \rP}) = \mathfrak{g}
    \end{equation}
    and all the higher $\Ext$-groups vanish as in \eqref{eq:vanishing-higher-exts-intro}.

    \medskip

    \item The sequence of vector bundles of the form
    \begin{multline}\label{eq:hypothetical-lefschetz-collection-intro}
      \hspace{50pt} \cO_{\rG / \rP}, \widetilde{T}_{\rG / \rP}, \cO_{\rG / \rP}(1), \widetilde{T}_{\rG / \rP}(1), \dots , \cO_{\rG / \rP}(s-1), \widetilde{T}_{\rG / \rP}(s-1), \\
      % \dots \cO_{\rG / \rP}(s-1), \widetilde{T}_{\rG / \rP}(s-1),
      \cO_{\rG / \rP}(s), \dots, \cO_{\rG / \rP}(r-1)
    \end{multline}
    where $r$ is the index of $\rG / \rP$ and $s \in [1, r]$ is the maximal integer defined by the condition
    \begin{equation*}
      \Ext^\bullet(\widetilde{T}_{\rG / \rP}(i), \widetilde{T}_{\rG / \rP}) = 0 \quad \text{for all} \quad 0 < i < s,
    \end{equation*}
    is a Lefschetz collection for any $\rG / \rP$ with the following exceptions:
    \begin{equation}
      \begin{aligned}
        & \text{all cases from \eqref{eq:list-of-exceptions-1-intro} and all quadrics.}
      \end{aligned}
    \end{equation}
    Note that all quadrics are given by the infinite series $(\rB_n, 1)$ for $n \geq 2$
    and $(\rD_n, 1)$ for $n \geq 4$, but also small rank sporadic cases $(\rA_3, 2)$,
    $(\rB_3, 3)$, $(\rC_2, 2)$, $(\rD_4, 3)$, $(\rD_4, 4)$, and $(\rG_2, 1)$.

    Moreover, the cases
    \begin{equation}
      \begin{aligned}
        & (\rB_n, 2) \text{  and  } (\rD_n, 2) \text{  for  } n \geq 4; \\
        & (\rA_5, k) \text{  for  } k \in [2,4]; \\
        & (\rC_3, 3)
      \end{aligned}
    \end{equation}
    are the only cases, where the collection \eqref{eq:hypothetical-lefschetz-collection-intro} is
    a non-rectangular Lefschetz collection, i.e. $s \neq r$.
  \end{enumerate}

\bigskip

  At this point we do not have rigorous proofs of the above claims. We hope to be
  able to say more on this in a future work.
\end{observation}

\begin{remark}
  Asking for the Yoneda map \eqref{eq:yoneda-map-intro} to be an isomorphism
  is very close to saying that the pair of vector bundles
  $\{ \cO_{\rG / \rP}, T_{\rG / \rP} \}$ forms an \textsf{exceptional block} in the sense of Kuznetsov and Polishchuk \cite[Definition 3.1]{KuPo}.
  The only difference is that in \cite{KuPo}
  exceptional blocks always consist of irreducible vector bundles,
  but the tangent bundle $T_{\rG / \rP}$
  is irreducible only for cominuscule $\rG / \rP$.
\end{remark}

The paper is organised as follows. In \S\ref{section:preliminaries} we recall the necessary general background
on $\rG$-equivariant vector bundles on $\rG/\rP$ and collect some auxiliary facts for our particular
setting. In \S\ref{section:construction} we construct and prove exceptionality of
\eqref{eq:collection-F4-P1-intro}. Finally, in \S\ref{section:fullness} we show that
\eqref{eq:collection-F4-P1-intro} is full.

\medskip

\noindent {\bf Acknowledgements.} First and foremost I would like to thank Alexander Kuznetsov
for his continuing support and mathematical generosity. Our innumerable discussions have greatly
influenced this work. Further, I am very grateful to Michel Brion and Nicolas Perrin for
explaining to me the proof of Lemma~\ref{lemma:B4-and-D4-as-linear-sections}. Finally,
I would like to thank HIM and MPIM in Bonn for the great working conditions during the
preparation of this paper.

% \newpage
\section{Preliminaries}
\label{section:preliminaries}

We work over a fixed algebraically closed field $\Bbbk$ of characteristic zero.

\subsection{Generalities on equivariant vector bundles}
\label{subsection:generalities-on-equivarian-vector-bundles}

Let $\rG$ be a connected simply connected simple algebraic group and $\rP$ its parabolic subgroup.
Let us also fix a maximal torus~$\rT$ and a Borel subgroup $\rB$ in $\rG$ such that we have inclusions
\begin{equation*}
  \rT \subset \rB \subset \rP \subset \rG
\end{equation*}
The Borel subgroup $\rB$ fixes a choice of \emph{negative roots} of $\rG$ with respect to $\rT$,
i.e. the root subspaces of $\rB$ correspond to negative roots.
We also have the positive Borel subgroup $\rB^+$ such that
\begin{equation*}
  \rT = \rB \cap \rB^+,
\end{equation*}
and whose roots spaces correspond to positive roots.

Finally, let $\rL \subset \rP$ be the Levi subgroup of $\rP$ containing the maximal
torus~$\rT$. Note that the intersections $\rB \cap \rL$ and $\rB^+ \cap \rL$ are
negative and positive Borel subgroups in $\rL$ respectively.

Recall that there is a monoidal equivalence of categories between the category of $\rG$-equivariant
coherent sheaves on $\rG/\rP$ and the category of finite dimensional representations of $\rP$
\begin{equation}\label{eq:equivalence-homogeneous-bundles-P-reps}
  \Coh^{\rG}(\rG/\rP) \overset{\cong}{\rightarrow} \Rep(\rP)
\end{equation}
given by sending a $\rG$-equivariant sheaf $F$ to its fiber $F_{[\rP]}$ at the point $[\rP] \in \rG/\rP$.
Since the group $\rP$ is not reductive, the category $\Rep(\rP)$ is not semisimple and working with it can
be quite hard. However, any irreducible representation of $\rP$ is completely determined by its restriction
to $\rL$, as the unipotent radical of $\rP$ acts trivially on it. Hence, we can identify irreducible representations
of $\rP$ and $\rL$. Since $\rL$ is a reductive group, we have the highest weight theory describing its finite dimensional
irreducible representations. This makes working with irreducible representations of $\rP$ much easier.

We denote by $\bfP_{\rL} = \bfP_{\rG}$ the weight lattices for $\rL$ and $\rG$ with their natural identification and by
\begin{equation*}
  \bfP_{\rG}^+ \subset \bfP_{\rL}^+
\end{equation*}
the cones of dominant weights (with respect to $\rB^+$ and $\rB^+ \cap \rL$ respectively).
For each dominant weight $\lambda \in \bfP_{\rL}^+$
we denote by $\rV_{\rL}^\lambda$ the corresponding irreducible representation of
$\rL$ and by $\cU^\lambda$ the $\rG$-equivariant vector bundle on $\rG/\rP$
corresponding to it via the equivalence \eqref{eq:equivalence-homogeneous-bundles-P-reps}.
Similarly, for $\lambda \in \bfP_{\rG}^+$ we denote by $\rV_{\rG}^\lambda$ the
corresponding irreducible representation of $\rG$.

We denote by $\rW$ the Weyl group of $\rG$ and by $\rW^{\rL} \subset \rW$ the Weyl group of $\rL$. We denote by $w_0 \in \rW$ and $w_0^{\rL} \in \rW^{\rL}$ the longest elements.

We have the following lemma.
\begin{lemma}\label{lemma:tensor-and-dual-general}
  \
  \begin{enumerate}
    \item We have $(\cU^\lambda)^\vee \cong \cU^{-w_0^\rL \lambda}$.

    \item If $\rV_\rL^\lambda \otimes \rV_\rL^\mu = \bigoplus \rV_\rL^\nu$, then $\cU^\lambda \otimes \cU^\mu = \bigoplus \cU^\nu$.
  \end{enumerate}
\end{lemma}

\begin{proof}
  This is an immediate consequence of the equivalence \eqref{eq:equivalence-homogeneous-bundles-P-reps}
  and \cite[Formula (6)]{KuPo}.
\end{proof}

For equivariant vector bundles of the form $\cU^\lambda$ we also have a very efficient way
to compute their cohomology. Let $\ell \colon \rW \to \bZ$ be the length function and let
$\rho \in \bfP_{\rG}^+$ be the sum of fundamental weights of $\rG$.

\begin{theorem}[Borel--Weil--Bott]
Consider a vector bundle $\cU^\lambda$ with $\lambda \in \bfP_{\rL}^+$. If the weight $\lambda + \rho$ lies on a wall of a Weyl chamber for the $\rW$-action, then
\begin{equation*}
  H^\bullet(\rG/\rP, \cU^\lambda) = 0.
\end{equation*}
Otherwise, if $w \in \rW$ is the unique element such that the weight $w(\lambda + \rho)$ is dominant, then
\begin{equation*}
  H^\bullet(\rG/\rP, \cU^\lambda) = \rV_{\rG}^{w(\lambda + \rho) - \rho}[-\ell(w)].
\end{equation*}
\end{theorem}

\subsection{Specifics of our case}

As in Section \ref{section:introduction} we let $\rG$ be a connected simple algebraic
group of Dynkin type $\rF_4$ and $\rP = \rP_1 \subset \rG$ be a maximal parabolic subgroup
associated with the first vertex of its Dynkin diagram \eqref{diagram:F4}.
We then consider the homogenous space
\begin{equation*}
  X = \rG/\rP,
\end{equation*}
which is the adjoint Grassmannian in type $\rF_4$.

The Levi subgroup $\rL$ of $\rP$ is a reductive group, whose semisimple part is
of Dynkin type~$\rC_3$. A more detailed description of $\rL$ is given below
in Section \ref{subsection:tensor-products-and-exterior-powers}.

We realize the weight lattice $\bfP_{\rG}$ of $\rG$ inside $\bQ^4$ as in \cite[p.211--213]{Bourbaki}.
Let $e_1, e_2, e_3, e_4$ be the standard basis of $\bQ^4$; it is orthonormal with respect to the usual
scalar product. The weight lattice $\bfP_{\rG}$ is generated inside $\bQ^4$ by the fundamental weights
$\omega_1, \dots, \omega_4$ or by the simple roots $\alpha_1, \dots, \alpha_4$, which are given by
\begin{equation}\label{eq:fundamental-weights-F4}
  \begin{aligned}
    & \omega_1 = e_1 + e_2 \\
    & \omega_2 = 2 e_1 + e_2 + e_3 \\
    & \omega_3 = \frac{1}{2}\left( 3 e_1 + e_2 + e_3 + e_4 \right) \\
    & \omega_4 = e_1
  \end{aligned}
  \hspace{70pt}
  \begin{aligned}
    & \alpha_1 = e_2 - e_3 \\
    & \alpha_2 = e_3 - e_4 \\
    & \alpha_3 = e_4 \\
    & \alpha_4 = \frac{1}{2}\left( e_1 - e_2 - e_3 - e_4 \right)
  \end{aligned}
\end{equation}
From \eqref{eq:fundamental-weights-F4} we see that a weight $\lambda = (\lambda_1, \lambda_2, \lambda_3, \lambda_4)$ is $\rG$-dominant if and only if
\begin{equation}\label{eq:dominant-cone-F4}
  \lambda_2 \geq \lambda_3 \geq \lambda_4 \geq 0 \quad \text{and} \quad \lambda_1 \geq \lambda_2 + \lambda_3 + \lambda_4 \, .
\end{equation}
The sum of the fundamental weights appearing in the Borel--Weil--Bott theorem is
\begin{equation}\label{eq:rho-F4}
  \rho = \frac{11}{2} e_1 + \frac{5}{2} e_2 + \frac{3}{2} e_3 + \frac{1}{2} e_4.
\end{equation}
All roots of $\rG$ are given by
\begin{equation}\label{eq:roots-F4}
  \pm e_i, \quad \pm e_i \pm e_j, \quad \frac{1}{2}(\pm e_1 \pm e_2 \pm e_3 \pm e_4)
\end{equation}
and, therefore, the walls of the Weyl chambers are given by the hyperplanes
\begin{equation}\label{eq:weyl-chambers-F4}
  \begin{aligned}
    & \lambda_i = 0, \quad 1 \leq i \leq 4, \\
    & \lambda_i = \pm \lambda_j , \quad 1 \leq i \neq j \leq 4, \\
    & \sum_{i \in I} \lambda_i = \sum_{j \in J} \lambda_j, \quad I \sqcup J = \{1,2,3,4\}.
  \end{aligned}
\end{equation}
Note that in the last condition $I$ and $J$ can be empty.

The longest element $w_0 \in \rW$ acts as
\begin{equation}\label{eq:longest-element-F4}
  w_0(\lambda_1, \lambda_2, \lambda_3, \lambda_4) = (-\lambda_1, -\lambda_2, -\lambda_3, -\lambda_4).
\end{equation}
Hence, we see that any representation of $\rG$ is self-dual. The longest element $w_0^{\rL} \in \rW^{\rL}$ acts by
\begin{equation}\label{eq:longest-element-levi}
  w_0^{\rL}(\lambda_1, \lambda_2, \lambda_3, \lambda_4) = (\lambda_2, \lambda_1, -\lambda_3, -\lambda_4).
\end{equation}

The ample generator of the Picard group is denoted by $\cO(1)$ and in terms of the equivalence~
\eqref{eq:equivalence-homogeneous-bundles-P-reps} we have
\begin{equation}\label{eq:line-bundle-twists}
  \cO(t) = \cU^{t\omega_1} \quad \text{and} \quad \cU^{\lambda + t \omega_1} = \cU^{\lambda}(t) \quad \text{for} \quad t \in \bZ.
\end{equation}

Now we can apply Lemma \ref{lemma:tensor-and-dual-general}(1) to compute the necessary duals.
\begin{corollary}\label{corollary:duals}
  We have
  \begin{equation*}
    \left( \cU^{a_2\omega_2 + a_3\omega_3 + a_4\omega_4} \right)^\vee =
    \cU^{a_2\omega_2 + a_3\omega_3 + a_4\omega_4}(- 3a_2 - 2a_3 - a_4).
  \end{equation*}
\end{corollary}

\begin{proof}
  Using \eqref{eq:longest-element-levi} and \eqref{eq:fundamental-weights-F4} we get
  \begin{multline*}
    -w_0^{\rL} (a_2\omega_2 + a_3\omega_3 + a_4\omega_4) = (-a_2 - \frac{1}{2}a_3, -2a_2 - \frac{3}{2}a_3 - a_4 , a_2 + \frac{1}{2}a_3, \frac{1}{2}a_3) = \\
    = (- 3a_2 - 2a_3 - a_4) \omega_1 + a_2\omega_2 + a_3\omega_3 + a_4\omega_4,
  \end{multline*}
  and the claim now follows from Lemma \ref{lemma:tensor-and-dual-general}(1) and \eqref{eq:line-bundle-twists}.
\end{proof}

\subsection{Tensor products and exterior powers}
\label{subsection:tensor-products-and-exterior-powers}

To compute tensor products of equivariant vector bundles on a homogeneous space
$\rG/\rP$ one can use Lemma \ref{lemma:tensor-and-dual-general}(2), which turns
it into the problem of computing tensor products of representations of the Levi
subgroup $\rL$. We explain here how to further reduce such computations to the known
results on tensor products for the semisimple part of $\rL$. Such results exist
in the literature in the form of tables (e.g.~\cite{VO}) and in the form of software
(e.g.~\cite{sagemath}). Essentially this is a consistent way to control the twist
by the center of $\rL$.

Let us temporarily go back into the general setting of Section \ref{subsection:generalities-on-equivarian-vector-bundles}. We denote by $\alpha_1, \dots, \alpha_n$ and $\omega_1, \dots, \omega_n$ the simple
roots and the fundamental weights of an arbitrary semisimple group $\rG$. We also make
use of the $\rW$-invariant scalar product on the ambient space of the root system of $\rG$, which we identify
with $\left( \bfP_{\rG} \right)_{\bQ}$; we denote it by $(\lambda, \mu)$. To simplify the notation, let us also
assume that the parabolic $\rP$ is maximal and corresponds to the $k$-th vertex of the Dynkin diagram of $\rG$.
Denote by $R(\rL)$ the roots of $\rL$, i.e. those roots of $\rG$, whose expression as a linear combination of
the simple roots $\alpha_1, \dots, \alpha_n$ does not involve the simple root $\alpha_k$. In other words,
$R(\rL)$ consists of those roots of $\rG$ that are orthogonal to $\omega_k$ with respect to the scalar product.

Consider two closed algebraic subgroups of $\rL$: the derived group $\rLprime \subset \rL$ and the center $Z(\rL) \subset \rL$. The derived group $\rLprime$ is a connected simply connected semisimple algebraic group, whose Dynkin diagram is obtained from the Dynkin diagram of $\rG$ by removing the $k$-th vertex. For the center we have an isomorphism $Z(\rL) \simeq \bG_m$. The group law in $\rL$ gives a surjective homomorphism of algebraic groups
\begin{equation}\label{eq:derived-subgroup-and-center}
  \rLprime \times Z(\rL) \to \rL
\end{equation}
with finite kernel $\rLprime \cap Z(\rL)$. The weight lattices of $\rLprime$ and
$Z(\rL)$ are the quotients of $\bfP_{\rL} = \bfP_{\rG}$
\begin{equation}\label{eq:maps-of-weight-lattices}
  \begin{aligned}
  \xymatrix{
    & \bfP_{\rL} \ar[dl]_{\varphi_{\rLprime}} \ar[dr]^{\varphi_Z} \\
    \bfP_{\rLprime} = \bfP_{\rL}/\left( \bZ\omega_k \right) & & \bfP_{Z(\rL)} = \bfP_{\rL}/\left(\sum_{\alpha \in R(\rL)} \bQ \alpha \cap \bfP_{\rL} \right)
  }
  \end{aligned}
\end{equation}
and the map
\begin{equation}\label{eq:embedding-of-weight-lattices}
  \bfP_{\rL} \xrightarrow{(\varphi_{\rLprime}, \varphi_{Z})} \bfP_{\rLprime} \oplus \bfP_{Z(\rL)}
\end{equation}
is an embedding as a finite index subgroup. Moreover, under $\varphi_{\rLprime}$ the fundamental
weights $\omega_1, \dots, \omega_n$ of $\rG$ get mapped to the fundamental weights
$\omega_1', \dots, \omega_{k-1}', \omega_{k+1}', \dots , \omega_n'$ of $\rLprime$
(with the labelling of simple roots inherited from $\rG$).
More precisely, we have
\begin{equation}\label{eq:maps-of-fundamental-weights-derived-subgroup}
  \varphi_{\rLprime}(\omega_i) =
  \begin{cases}
    \omega_i' & i \neq k, \\
    0 & i = k.
  \end{cases}
\end{equation}
A more detailed discussion of this setup can be found in \cite[II.1.18]{Jantzen}.

Over $\bQ$ we have an orthogonal direct sum decomposition
\begin{equation*}
  \left( \bfP_{\rL} \right)_{\bQ} \cong \left(\sum_{\alpha \in R(\rL)} \bQ \alpha \right) \oplus \left( \bQ \omega_k \right),
\end{equation*}
that allows us to identify
\begin{equation*}
  \left( \bfP_{\rLprime} \right)_{\bQ} \cong \sum_{\alpha \in R(\rL)} \bQ \alpha \qquad \text{and} \qquad \left( \bfP_{Z(\rL)} \right)_{\bQ} \cong \bQ \omega_k.
\end{equation*}
Thus, after tensoring with $\bQ$ the map \eqref{eq:embedding-of-weight-lattices} becomes an isomorphism
\begin{equation}\label{eq:decomposition-of-weights-map}
  \left( \bfP_{\rL} \right)_\bQ \xrightarrow{(\varphi_{\rLprime}, \varphi_{Z})} \left( \bfP_{\rLprime} \right)_\bQ \oplus \bQ \omega_k
\end{equation}
such that for a $\lambda = \sum_i a_i \omega_i \in \left( \bfP_{\rL} \right)_\bQ$ we have
\begin{equation}\label{eq:restriction-explicit-formula}
  \begin{aligned}
    & \varphi_{\rLprime}(\lambda) = \lambda' \, , \quad \text{where} \quad \lambda' = \sum_{i \neq k} a_i \omega_i', \\
    & \varphi_{Z}(\lambda) = r_{\lambda}\omega_k \, , \quad \text{where} \quad r_\lambda = \frac{(\lambda, \omega_k)}{(\omega_k, \omega_k)}.
  \end{aligned}
\end{equation}

Note that the map
\begin{equation*}
  \begin{aligned}
    & \varphi_{\rLprime} \colon \bfP_{\rL} \to \bfP_{\rLprime} \\
    & \qquad \lambda \hspace{5pt} \mapsto \hspace{5pt} \lambda'
  \end{aligned}
\end{equation*}
is completely determined by \eqref{eq:maps-of-fundamental-weights-derived-subgroup}, as $\omega_1, \dots, \omega_n$ form a basis of $\bfP_{\rL}$. We can also consider the map
\begin{equation}\label{eq:lifting-from-derived-subgroup}
  \begin{aligned}
    \bfP_{\rLprime} & \to \bfP_{\rL} \\
    \sum_{i \neq k} a_i \omega_i' & \mapsto \hspace{5pt} \sum_{i \neq k} a_i \omega_i
  \end{aligned}
\end{equation}
to which we will refer to as a \textsf{lifting} from $\bfP_{\rLprime}$ to $\bfP_{\rL}$.

Now we are ready to discuss tensor products. By \eqref{eq:derived-subgroup-and-center}
a representation of $\rL$ corresponds to a pair of compatible representations of $\rLprime$
and $Z(\rL)$. We are going to exploit this to compute tensor operations for
representations of $\rL$. For a highest weight irreducible representation $V_{\rL}^{\lambda}$ its restrictions $\Res^{\rL}_{\rLprime}(V_{\rL}^{\lambda})$ and $\Res^{\rL}_{Z(\rL)}(V_{\rL}^{\lambda})$ are again irreducible.
Hence, we have
\begin{equation*}
  \Res^{\rL}_{\rLprime}(V_{\rL}^{\lambda}) = V_{\rLprime}^{\lambda'},
\end{equation*}
where $\lambda'$ is the image of $\lambda$ in $\bfP_{\rLprime}$ under
\eqref{eq:maps-of-fundamental-weights-derived-subgroup}, and $\Res^{\rL}_{Z(\rL)}(V_{\rL}^{\lambda})$
is given by a character of $Z(\rL) \simeq \bG_m$.

\begin{lemma}\label{lemma:tensor-exterior-symmetric}
  Let $V_{\rL}^{\lambda}, \, V_{\rL}^{\mu}$ be irreducible representations of $\rL$
  with highest weights $\lambda, \mu \in \bfP_{\rL}$. Let $V_{\rLprime}^{\lambda'}, \, V_{\rLprime}^{\mu'}$
  be the induced irreducible representations of $\rLprime$ as above.

  \begin{enumerate}
    \item If
      \begin{equation*}\label{eq:tensor-product-derived-subgroup}
        V_{\rLprime}^{\lambda'} \otimes V_{\rLprime}^{\mu'} = \bigoplus_{\nu' \in \Sigma} \left( V_{\rLprime}^{\nu'} \right)^{m(\lambda', \mu' , \nu')}
      \end{equation*}
      then
      \begin{equation*}\label{eq:tensor-product-levi}
        V_{\rL}^{\lambda} \otimes V_{\rL}^{\mu} =
        \bigoplus_{\nu' \in \Sigma} \left( V_{\rL}^{\nu + (r_{\lambda} + r_{\mu} - r_{\nu})\omega_k} \right)^{m(\lambda', \mu' , \nu')}
      \end{equation*}
      where for a weight $\nu' \in \Sigma$ we denote by $\nu$ its lifting \eqref{eq:lifting-from-derived-subgroup}.

    \item If
    \begin{equation*}\label{eq:exterior-power-derived-subgroup}
      \Lambda^q \, V_{\rLprime}^{\lambda'} = \bigoplus_{\nu' \in \Sigma} \left( V_{\rLprime}^{\nu'} \right)^{m(\lambda', \nu')}
    \end{equation*}
    then
    \begin{equation*}\label{eq:exterior-power-levi}
      \Lambda^q \, V_{\rL}^{\lambda} =
      \bigoplus_{\nu' \in \Sigma} \left( V_{\rL}^{\nu + (qr_{\lambda} - r_{\nu})\omega_k} \right)^{m(\lambda', \nu')}
    \end{equation*}
    where for a weight $\nu' \in \Sigma$ we denote by $\nu$ its lifting \eqref{eq:lifting-from-derived-subgroup}.

    \item If
    \begin{equation*}\label{eq:symmetric-power-derived-subgroup}
      S^q \, V_{\rLprime}^{\lambda'} = \bigoplus_{\nu' \in \Sigma} \left( V_{\rLprime}^{\nu'} \right)^{m(\lambda', \nu')}
    \end{equation*}
    then
    \begin{equation*}\label{eq:symmetric-power-levi}
      S^q \, V_{\rL}^{\lambda} =
      \bigoplus_{\nu' \in \Sigma} \left( V_{\rL}^{\nu + (qr_{\lambda} - r_{\nu})\omega_k} \right)^{m(\lambda', \nu')}
    \end{equation*}
    where for a weight $\nu' \in \Sigma$ we denote by $\nu$ its lifting \eqref{eq:lifting-from-derived-subgroup}.
  \end{enumerate}
\end{lemma}

\begin{proof}
  Let us only prove the first statement, as the others are proved in a very similar way.

  As discussed just before the statement of the lemma, to compute the tensor product
  $V_{\rL}^{\lambda} \otimes V_{\rL}^{\mu}$ it is enough to determine how $\rLprime$ and
  $Z(\rL)$ act on it. Restricting to $\rLprime$ we need to compute the tensor product
  $V_{\rLprime}^{\lambda'} \otimes V_{\rLprime}^{\mu'}$, which decomposes as
  \begin{equation*}
    V_{\rLprime}^{\lambda'} \otimes V_{\rLprime}^{\mu'} = \bigoplus_{\nu' \in \Sigma} \left( V_{\rLprime}^{\nu'} \right)^{m(\lambda', \mu' , \nu')}
  \end{equation*}
  by our assumptions. Therefore, $V_{\rL}^{\lambda} \otimes V_{\rL}^{\mu}$ has a similar
  direct sum decomposition
  \begin{equation*}
    V_{\rL}^{\lambda} \otimes V_{\rL}^{\mu} =
    \bigoplus_{\nu' \in \Sigma} \left( V_{\rL}^{\nu + t_{(\lambda, \mu, \nu)} \omega_k} \right)^{m(\lambda', \mu' , \nu')},
  \end{equation*}
  where for a $\nu' \in \Sigma$ we denote by $\nu$ its lifting \eqref{eq:lifting-from-derived-subgroup} and
  $t_{(\lambda, \mu, \nu)}$'s are some integers. The integers $t_{(\lambda, \mu, \nu)}$
  are determined by looking at the action of the center $Z(\rL)$. Since on each
  $V_{\rL}^{\nu + t_{(\lambda, \mu, \nu)} \omega_k}$ the torus $Z(\rL)$ acts by a
  character, and since characters simply add under the tensor product,
  we get the relation
  \begin{equation*}
    r_{\nu + t_{(\lambda, \mu, \nu)}\omega_k} = r_{\lambda} + r_{\mu},
  \end{equation*}
  which implies the desired $t_{(\lambda, \mu, \nu)} = r_{\lambda} + r_{\mu} - r_{\nu}$.
\end{proof}

\medskip

From now on we return back to the case of our adjoint variety in type $\rF_4$.
In Example \ref{example:tensor-product} we illustrate how to use Lemma \ref{lemma:tensor-exterior-symmetric}
in practice and then in Lemmas \ref{lemma:tensor-products-F4-P1} and \ref{lemma:exterior-powers-omega_4}
we collect all tensor products and exterior/symmetric powers required for this paper.

Using \eqref{eq:fundamental-weights-F4} and \eqref{eq:restriction-explicit-formula} we get
\begin{equation*}
  r_{\omega_1} = 1, \quad r_{\omega_2} = \frac{3}{2}, \quad  r_{\omega_3} = 1, \quad r_{\omega_4} = \frac{1}{2}.
\end{equation*}

\begin{example}
  \label{example:tensor-product}
  Let us compute $\cU^{\omega_4}(-1) \otimes \cU^{\omega_4}(-1)$ using
  Lemma \ref{lemma:tensor-exterior-symmetric}, i.e. we need to compute
  $V_{\rL}^{\omega_4 - \omega_1} \otimes V_{\rL}^{\omega_4 - \omega_1}$.
  In the notation of Lemma \ref{lemma:tensor-exterior-symmetric} we have
  \begin{equation*}
    \lambda = \mu = \omega_4 - \omega_1 \quad \text{and} \quad \lambda' = \mu' = \omega_4' \quad \text{and} \quad
    r_\lambda = r_\mu = -\frac{1}{2}.
  \end{equation*}
  Using the table in \cite[p.302]{VO} we get
  \begin{equation*}
    V_{\rLprime}^{\omega_4'} \otimes V_{\rLprime}^{\omega_4'} = \Bbbk \oplus V_{\rLprime}^{\omega_3'} \oplus V_{\rLprime}^{2\omega_4'}.
  \end{equation*}
  Thus, we have $\Sigma = \{ 0 \, , \omega_3' \, , 2\omega_4' \}$
  and all multiplicities $m(\lambda', \mu' , \nu')$ are equal to $1$.
  Therefore, applying Lemma \ref{lemma:tensor-exterior-symmetric}(1) we get
  \begin{equation*}
    V_{\rL}^{\omega_4 - \omega_1} \otimes V_{\rL}^{\omega_4 - \omega_1} = V_{\rL}^{-\omega_1} \oplus
    V_{\rL}^{\omega_3 - 2 \omega_1} \oplus V_{\rL}^{2\omega_4 - 2 \omega_1},
  \end{equation*}
  and, finally,
  \begin{equation*}
    \cU^{\omega_4}(-1) \otimes \cU^{\omega_4}(-1) = \cO(-1) \oplus \cU^{\omega_3}(-2) \oplus \cU^{2\omega_4}(-2).
  \end{equation*}
\end{example}

\begin{lemma}
  \label{lemma:tensor-products-F4-P1}
  We have
  \begin{equation*}
    \begin{aligned}
      \cU^{\omega_4} \otimes \cU^{\omega_4} & = \cU^{2\omega_4} \oplus \cU^{\omega_3} \oplus \cO(1) \\
      \cU^{\omega_4} \otimes \cU^{\omega_3} & = \cU^{\omega_3 + \omega_4} \oplus \cU^{\omega_2} \oplus \cU^{\omega_4}(1) \\
      \cU^{\omega_4} \otimes \cU^{\omega_2} & = \cU^{\omega_2 + \omega_4} \oplus \cU^{\omega_3}(1) \\
      \cU^{\omega_2} \otimes \cU^{\omega_3} & = \cU^{\omega_2 + \omega_3} \oplus \cU^{\omega_3 + \omega_4}(1) \oplus \cU^{\omega_4}(2) \\
      \cU^{\omega_2} \otimes \cU^{\omega_2} & = \cU^{2\omega_2} \oplus \cU^{2\omega_3}(1) \oplus \cU^{2\omega_4}(2) \oplus \cO(3) \\
      \cU^{2\omega_4} \otimes \cU^{\omega_4} & = \cU^{3\omega_4} \oplus \cU^{\omega_3 + \omega_4} \oplus \cU^{\omega_4}(1) \\
      \cU^{2\omega_4} \otimes \cU^{\omega_3} & = \cU^{\omega_3 + 2\omega_4} \oplus \cU^{\omega_2 + \omega_4} \oplus \cU^{2\omega_4}(1) \oplus \cU^{\omega_3}(1) \\
      \cU^{2\omega_4} \otimes \cU^{\omega_2} & = \cU^{\omega_2 + 2\omega_4} \oplus \cU^{\omega_3 + \omega_4}(1) \oplus \cU^{\omega_2}(1)
    \end{aligned}
  \end{equation*}
\end{lemma}

\begin{proof}
  The proof is a routine calculation similar to Example \ref{example:tensor-product} using \cite[p.302]{VO} and Lemma~\ref{lemma:tensor-exterior-symmetric}.
\end{proof}

\begin{lemma}\label{lemma:exterior-powers-omega_4}
  We have
  \begin{equation*}
    S^k \cU^{\omega_4} = \cU^{k\omega_4} \quad \forall k \qquad \text{and} \qquad
    \Lambda^k \cU^{\omega_4} =
    \begin{cases}
      \cO & k = 0 \\
      \cU^{\omega_4} & k = 1 \\
      \cU^{\omega_3} \oplus \cO(1) & k = 2 \\
      \cU^{\omega_2} \oplus \cU^{\omega_4}(1) & k = 3 \\
      \cU^{\omega_3}(1) \oplus \cO(2) & k = 4 \\
      \cU^{\omega_4}(2) & k = 5 \\
      \cO(3) & k = 6
    \end{cases}
  \end{equation*}
  \begin{equation*}
    \Lambda^2 \cU^{\omega_3} = \cU^{\omega_2 + \omega_4} \oplus \cU^{2\omega_4}(1)
  \end{equation*}
  \begin{equation*}
    \Lambda^3 \cU^{\omega_3} = \cU^{2\omega_2} \oplus \cU^{\omega_3 + 2\omega_4}(1) \oplus \cU^{\omega_2 + \omega_4}(1) \oplus \cU^{2\omega_4}(2)
  \end{equation*}
\end{lemma}

\begin{proof}
  The proof is a routine calculation similar to Example \ref{example:tensor-product} using \cite[p.302]{VO} and Lemma~\ref{lemma:tensor-exterior-symmetric}.
\end{proof}

\subsection{Cohomology computations}

\begin{lemma}
  We have
  \begin{equation}\label{eq:cohomology-structure-sheaf}
    H^\bullet (X, \cO) = \Bbbk[0],
  \end{equation}

  \begin{equation}\label{eq:cohomology-omega_2(-2)}
    H^\bullet(X, \cU^{\omega_2}(-2)) = \Bbbk[-1],
  \end{equation}

  \begin{equation}\label{eq:cohomology-2omega_2(-3)}
    H^\bullet(X, \cU^{2\omega_2}(-3)) = V^{\omega_1}[-1].
  \end{equation}

\end{lemma}

\begin{proof}

The equality \eqref{eq:cohomology-structure-sheaf} follows from the fact that
$X$ is a smooth connected Fano variety together with the Kodaira vanishing theorem.

To prove \eqref{eq:cohomology-omega_2(-2)} and \eqref{eq:cohomology-2omega_2(-3)}
we apply the Borel--Weil--Bott theorem. Indeed, in the first case we have
$\lambda = \omega_2 - 2 \omega_1$ and it is easy to see using \eqref{eq:fundamental-weights-F4}--\eqref{eq:rho-F4}
that we have $s_{\alpha_1}(\lambda + \rho) = \rho$, hence \eqref{eq:cohomology-omega_2(-2)} holds. Similarly,
in the second case we have $\lambda = 2\omega_2 - 3 \omega_1$ and
$s_{\alpha_1}(\lambda + \rho) = \omega_1 + \rho$, and so \eqref{eq:cohomology-2omega_2(-3)} follows.
\end{proof}

By the Borel--Weil--Bott theorem a vector bundle $\cU^{\lambda}$ is acyclic if and only if the weight $\lambda + \rho$ lies on a wall of a Weyl chamber for the $\rW$-action; such weights are called \textsf{singular}. A routine computation using formulas \eqref{eq:weyl-chambers-F4} for the walls of the Weyl chambers gives the following.

\begin{corollary}\label{corollary:acyclicities}
  The following vector bundles are acyclic
  \begin{equation*}
    \begin{aligned}
      \cO(t) \quad & \text{for} \quad t \in [-7,-1] \\
      \cU^{\omega_2}(t) \quad & \text{for} \quad t \in [-10,-1] \setminus \{-9, -2\} \\
      \cU^{\omega_3}(t) \quad & \text{for} \quad t \in [-9,-1] \\
      \cU^{\omega_4}(t) \quad & \text{for} \quad t \in [-8,-1] \\
      \cU^{2\omega_2}(t) \quad & \text{for} \quad t \in [-10,-1] \setminus \{-3, -2\} \\
      \cU^{2\omega_3}(t) \quad & \text{for} \quad t \in [-11,-1] \setminus \{-9, -3\} \\
      \cU^{2\omega_4}(t) \quad & \text{for} \quad t \in [-9,-1] \\
      \cU^{\omega_2 + \omega_4}(t) \quad & \text{for} \quad t \in [-11,-1] \setminus \{-10, -2\} \\
      \cU^{\omega_3 + \omega_4}(t) \quad & \text{for} \quad t \in [-10,-1] \\
      \cU^{\omega_2 + 2\omega_4}(t) \quad & \text{for} \quad t \in [-12,-1] \setminus \{-11, -2\} \\
      \cU^{3\omega_4}(t) \quad & \text{for} \quad t \in [-10,-1]
    \end{aligned}
  \end{equation*}
\end{corollary}

\begin{remark}\label{remark:acyclicity-twist-by-minus-one}
  Note that for any $\rL$-dominant weight
  \begin{equation*}
    \lambda = \sum_i a_i \omega_i \quad \text{with} \quad a_1 = -1
  \end{equation*}
  the vector bundle $\cU^\lambda$ is always acyclic. Indeed, this follows from the Borel--Weil--Bott theorem, as $\lambda + \rho$ is fixed by the simple reflection $s_{\alpha_1}$ and is, therefore, a singular weight.
\end{remark}

\section{Construction of the exceptional collection}
\label{section:construction}

Recall that $X$ is the adjoint Grassmannian of type $\rF_4$.
In this section we show that the sequence of vector bundles
\begin{equation}\label{eq:collection-F4-P1}
  \cO, \cU^{\omega_4}, \widetilde{T}_X, \cO(1), \cU^{\omega_4}(1), \widetilde{T}_X(1), \dots,
  \cO(7), \cU^{\omega_4}(7), \widetilde{T}_X(7)
\end{equation}
(where $\widetilde{T}_X$ is an extension of the tangent bundle $T_X$ defined below in \eqref{eq:exact-sequence-extension-tangent-bundle}) is an exceptional collection on $X$. As already mentioned in the Introduction, this collection is a rectangular Lefschetz exceptional collection.

Let us first consider the subcollection of \eqref{eq:collection-F4-P1} generated
by $\cO$ and $\cU^{\omega_4}$.

\begin{lemma}\label{lemma:collection-O-and-U}
  The collection of vector bundles
  \begin{equation*}
    \cO, \cU^{\omega_4}, \cO(1), \cU^{\omega_4}(1), \dots, \cO(7), \cU^{\omega_4}(7)
    \end{equation*}
  is exceptional.
\end{lemma}

\begin{proof}
  We first show that subcollections $\cO, \cO(1), \dots, \cO(7)$ and $\cU^{\omega_4}, \cU^{\omega_4}(1), \dots, \cU^{\omega_4}(7)$
  are exceptional, and then show that they can be merged together as in the statement of the lemma.

  By \eqref{eq:cohomology-structure-sheaf} any line bundle $\cO(t)$ is exceptional
  and from Corollary \ref{corollary:acyclicities} it follows that the subcollection $\cO, \cO(1), \dots, \cO(7)$
  is exceptional. Thus, we only need to deal with the subcollection
  $\cU^{\omega_4}, \cU^{\omega_4}(1), \dots, \cU^{\omega_4}(7)$. It is enough to show
  \begin{equation*}
    \Ext^\bullet(\cU^{\omega_4}(t), \cU^{\omega_4}) =
    \begin{cases}
      \Bbbk[0] & \text{if} \quad t = 0, \\
      0 & \text{if} \quad 1 \leq t \leq 7.
    \end{cases}
  \end{equation*}
  By Corollary \ref{corollary:duals} and Lemma \ref{lemma:tensor-products-F4-P1} we rewrite
  \begin{multline*}
    \Ext^\bullet(\cU^{\omega_4}(t), \cU^{\omega_4})
     % = H^\bullet(X, \cU^{\omega_4} \otimes \left( \cU^{\omega_4}(t) \right)^\vee)
    = H^\bullet(X, \cU^{\omega_4} \otimes \cU^{\omega_4}(-t-1)) = \\
    = H^\bullet(X,  \cU^{2\omega_4}(-t-1)) \oplus H^\bullet(X, \cU^{\omega_3}(-t-1)) \oplus H^\bullet(X, \cO(-t)).
  \end{multline*}
  Finally, applying \eqref{eq:cohomology-structure-sheaf} and Corollary \ref{corollary:acyclicities} we get the desired result.

  To merge the two subcollections it is enough to show
  \begin{equation*}
    \begin{aligned}
      & \Ext^\bullet(\cU^{\omega_4}(t), \cO) = 0 \quad \text{for} \quad 0 \leq t \leq 7, \\
      & \Ext^\bullet(\cO(t), \cU^{\omega_4}) = 0 \quad \text{for} \quad 1 \leq t \leq 7.
    \end{aligned}
  \end{equation*}
  Rewriting both $\Ext$-groups in terms of cohomology, and using Corollary \ref{corollary:duals}
  to dualize $\cU^{\omega_4}(t)$, we see that the above conditions are equivalent to
  \begin{equation*}
    H^\bullet(X, \cU^{\omega_4}(-t)) = 0 \quad \text{for} \quad 1 \leq t \leq 8,
  \end{equation*}
  which holds by Corollary \ref{corollary:acyclicities}.
\end{proof}

Now we start our discussion of the tangent bundle.
Recall that for any rational homogeneous space $\rG/\rP$ we have the following description
of the tangent bundle. Let $\mathfrak{g}$
and $\mathfrak{p}$ be the Lie algebras of $\rG$ and $\rP$ respectively. The group $\rG$
(resp. $\rP$) acts on the Lie algebra $\mathfrak{g}$ (resp. $\mathfrak{p}$) via the adjoint
representation. It is well-known that under the equivalence \eqref{eq:equivalence-homogeneous-bundles-P-reps}
the tangent bundle $T_{\rG/\rP}$ corresponds to the $\rP$-representation~$\mathfrak{g}/\mathfrak{p}$.
Using the Killing form, we get an isomorphism of $\rP$-representations
$\mathfrak{g}/\mathfrak{p} = \left( \mathfrak{n} \right)^\vee$, where~$\mathfrak{n}$ is
the nilradical of $\mathfrak{p}$. This way we also see that the cotangent bundle $\Omega_{\rG/\rP}$
corresponds to $\mathfrak{n}$ under \eqref{eq:equivalence-homogeneous-bundles-P-reps}. Finally,
we note that $T_{\rG/\rP}$ and $\Omega_{\rG/\rP}$ are irreducible if and only if the homogeneous
space $\rG/\rP$ is cominuscule.

Now we are ready to add the (extension of the) tangent bundle to the collection.
For a $\rG$-equivariant vector bundle $E$ we denote by $\ss(E)$ its semisimplification, i.e.
the direct sum of composition factors of $E$.

\begin{lemma}\label{lemma:definition-of-T_X_tilde}
  The following statements hold.
  \begin{enumerate}
    \item On $X$ there exist unique $\rG$-equivariant non-split short exact sequences
    \begin{equation}\label{eq:exact-sequence-tangent-bundle}
      0 \to \cU^{\omega_2}(-1) \to T_X \to \cO(1) \to 0
    \end{equation}
    and
    \begin{equation}\label{eq:exact-sequence-extension-tangent-bundle}
      0 \to \cO \to \widetilde{T}_X \to T_X \to 0,
    \end{equation}
    where $T_X$ is the tangent bundle of $X$ and $\widetilde{T}_X$ is some $\rG$-equivariant vector bundle of rank 16.

    \item We have
    \begin{equation}\label{eq:semisimplification-T_X_tilde}
      \ss(\widetilde{T}_X) = \cO \oplus \cU^{\omega_2}(-1) \oplus \cO(1).
    \end{equation}

    \item We have
    \begin{equation}\label{eq:T_X_tilde-as-mutation}
      T_X = \mathbf{L}_{\langle \cO(1) \rangle}(\cU^{\omega_2}(-1)) \quad \text{and} \quad \widetilde{T}_X = \mathbf{R}_{\langle  \cO \rangle}(T_X),
    \end{equation}
    where $\mathbf{L}$ and $\mathbf{R}$ stand for the left and right mutation functors.

    \item We have a $\rG$-equivariant isomorphism
    \begin{equation}\label{eq:dual-T_X_tilde}
      \widetilde{T}_X = (\widetilde{T}_X)^\vee(1).
    \end{equation}
  \end{enumerate}
\end{lemma}

\begin{proof}
  Consider the short exact sequence of $\rP$-representations
  \begin{equation*}
    0 \to [\mathfrak{n}, \mathfrak{n}] \to \mathfrak{n} \to \mathfrak{n}/[\mathfrak{n}, \mathfrak{n}] \to 0.
  \end{equation*}
  Since $X$ is an adjoint variety, $[\mathfrak{n}, \mathfrak{n}]$ is one-dimensional
  and $\mathfrak{n}/[\mathfrak{n}, \mathfrak{n}]$ is irreducible
  (see \cite[Section 3.3]{BS}, for example). More concretely, using the formulas
  \eqref{eq:fundamental-weights-F4}--\eqref{eq:roots-F4} to determine the $\rL$-dominant
  roots of $\mathfrak{n}$ we find
  \begin{equation*}
    - \alpha_1 = \omega_2 - 2\omega_1, \quad -\alpha_1 - \alpha_3 - \alpha_4 = \omega_4-\omega_1, \quad - 2 \alpha_1 - 3 \alpha_2 - 4 \alpha_3 - 2 \alpha_4 = - \Theta = - \omega_1.
  \end{equation*}
  Now it follows easily that the weight of $[\mathfrak{n}, \mathfrak{n}]$ is $-\omega_1$
  and, since $- \alpha_1 > -\alpha_1 - \alpha_3 - \alpha_4$  we obtain that $\mathfrak{n}/[\mathfrak{n}, \mathfrak{n}]$ is an irreducible $\rL$-representation
  with highest weight $\omega_2 - 2\omega_1$.
  Therefore,
  for the cotangent bundle we have a short exact sequence
  \begin{equation}\label{eq:exact-sequence-cotangent-bundle}
    0 \to \cO(-1) \to \Omega_X \to \cU^{\omega_2}(-2) \to 0.
  \end{equation}
  Dualizing \eqref{eq:exact-sequence-cotangent-bundle} and using Corollary \ref{corollary:duals}
  we get \eqref{eq:exact-sequence-tangent-bundle}.
  Since by \eqref{eq:cohomology-omega_2(-2)} we have
  \begin{equation}\label{eq:ext-first-extension}
    \Ext^\bullet(\cO(1), \cU^{\omega_2}(-1)) = \Bbbk[-1],
  \end{equation}
  there exists a unique non-trivial $\rG$-equivariant extension of the form \eqref{eq:exact-sequence-tangent-bundle}.

  \medskip

  To prove \eqref{eq:exact-sequence-extension-tangent-bundle} we note
  \begin{equation}\label{eq:ext-second-extension}
    \Ext^\bullet(T_X, \cO) = H^\bullet(X, \Omega_X) = H^\bullet(X, \cU^{\omega_2}(-2)) = \Bbbk[-1],
  \end{equation}
  where we used \eqref{eq:exact-sequence-cotangent-bundle}, Corollary \ref{corollary:acyclicities},
  and \eqref{eq:cohomology-omega_2(-2)}.
  Thus, there exists a unique non-trivial $\rG$-equivariant extension $\widetilde{T}_X$ as in \eqref{eq:exact-sequence-extension-tangent-bundle}.

  \medskip

  For \eqref{eq:semisimplification-T_X_tilde} we just use \eqref{eq:exact-sequence-tangent-bundle}
  and \eqref{eq:exact-sequence-extension-tangent-bundle}.

  \medskip

  For \eqref{eq:T_X_tilde-as-mutation} it is enough to show that \eqref{eq:exact-sequence-tangent-bundle}
  and \eqref{eq:exact-sequence-extension-tangent-bundle} are the corresponding mutation triangles.
  For the former it follows from \eqref{eq:ext-first-extension} and for the latter from \eqref{eq:ext-second-extension}.

  \medskip

  To show \eqref{eq:dual-T_X_tilde} it is convenient to use convolutions
  (see \cites{MR4452435, MR1465519} for background).
  The exact sequences \eqref{eq:exact-sequence-tangent-bundle}--\eqref{eq:exact-sequence-extension-tangent-bundle}
  imply that $\widetilde{T}_X[1]$ is the left convolution of the $3$-term complex
  \begin{equation}\label{eq:complex-ABC}
    \cO(1)[-1] \to \cU^{\omega_2}(-1) \to \cO[1],
  \end{equation}
  in the triangulated category $\Db(X)$; the composition is zero due to the vanishing
  \begin{equation}\label{eq:complex-ABC-ext-vanishing}
    \Ext^\bullet(\cO(1), \cO) = H^\bullet(X, \cO(-1)) = 0
  \end{equation}
  that holds by Corollary~\ref{corollary:acyclicities} and hence \eqref{eq:complex-ABC}
  is indeed a complex. By \eqref{eq:complex-ABC-ext-vanishing} convolutions of
  \eqref{eq:complex-ABC} are unique.
  Dualizing \eqref{eq:complex-ABC}
  and then twisting the result by $\cO(1)$ we get a new $3$-term complex,
  whose right convolution is $(\widetilde{T}_X)^\vee(1)$. However, using Corollary \ref{corollary:duals}
  we see that the new complex is nothing else but the old complex \eqref{eq:complex-ABC} and,
  therefore, its left convolution must be again $\widetilde{T}_X[1]$. Finally,
  since the left convolution always differs from the right convolution just by a shift by $[1]$, we obtain
  the desired claim.
\end{proof}

\begin{remark}
  The isomorphism \eqref{eq:dual-T_X_tilde} gives rise to a $\rG$-equivariant non-degenerate
  bilinear pairing. Moreover, one can show that this pairing is skew-symmetric, i.e. we have $\rG$-equivariant pairing
  \begin{equation*}
    \Lambda^2 \widetilde{T}_X \to \cO(1).
  \end{equation*}
  This pairing comes from the skew-symmetric $\rG$-equivariant pairing on $\cU^{\omega_2}(-1)$. As we won't need this fact, we omit the proof.
\end{remark}

Using \eqref{eq:exact-sequence-tangent-bundle}--\eqref{eq:exact-sequence-extension-tangent-bundle} we compute global sections
\begin{equation*}
  H^0(X, T_X) = V^{\omega_1} = \mathfrak{g} \quad \text{and} \quad H^0(X, \widetilde{T}_X) = V^{\omega_1} \oplus \bC.
\end{equation*}
Since $X$ is homogeneous, the tangent bundle $T_X$ is globally generated and we have a surjection
\begin{equation*}
  V^{\omega_1} \otimes \cO \to T_X.
\end{equation*}
Therefore, $\widetilde{T}_X$ is also globally generated and we have a surjection
\begin{equation}\label{eq:global-generation-Ttilde}
  \left( V^{\omega_1} \oplus \bC \right) \otimes \cO \xrightarrow{(\alpha, \beta)} \widetilde{T}_X.
\end{equation}
Recall that a 3-term complex sitting in degrees $-1$, $0$, $1$ and exact in the first and last terms is called a \textsf{monad}.

\begin{lemma}\label{lemma:monad-U}
There is a monad
\begin{equation}\label{eq:monad-U}
  \widetilde{T}_X(-1) \to \left( V^{\omega_1} \oplus \bC \right) \otimes \cO \xrightarrow{(\alpha, \beta)} \widetilde{T}_X,
\end{equation}
whose middle cohomology is isomorphic to $\cU^{2\omega_4}(-1)$. In particular, the kernel of \eqref{eq:global-generation-Ttilde} is an extension of $\widetilde{T}_X(-1)$ and $\cU^{2\omega_4}(-1)$.
\end{lemma}

\begin{proof}
Let us consider the diagram
\begin{equation}\label{eq:commutative-square}
  \xymatrix{
  \widetilde{T}_X(-1) \ar[r]^{\alpha^t} \ar[d]_{\beta^t} & V^{\omega_1} \otimes \cO \ar[d]^{\alpha} \\
  \cO \ar[r]^{\beta} & \widetilde{T}_X
  }
\end{equation}
Here $\alpha$ is the restriction of \eqref{eq:global-generation-Ttilde} to $V^{\omega_1} \otimes \cO$,
$\beta$ is provided by \eqref{eq:exact-sequence-extension-tangent-bundle} and $\alpha^t$, $\beta^t$ are
their duals; we also used \eqref{eq:dual-T_X_tilde} and the self-duality of $V^{\omega_1} \otimes \cO$.
Since all the maps in \eqref{eq:commutative-square} are $\rG$-equivariant, their compositions
$\beta \circ \beta^t$ and $\alpha \circ \alpha^t$ are also $\rG$-equivariant.
Moreover, using the equivalence \eqref{eq:equivalence-homogeneous-bundles-P-reps}
it is easy to see from some basic linear algerba that both $\beta \circ \beta^t$
and~$\alpha \circ \alpha^t$ are non-zero.

Using \eqref{eq:semisimplification-T_X_tilde}, \eqref{eq:dual-T_X_tilde}, Lemma \ref{lemma:tensor-products-F4-P1}
and Borel--Weil--Bott it is easy to see that the space of $\rG$-equivariant morphisms
$\Hom(\widetilde{T}_X(-1), \widetilde{T}_X) = H^0(X, \widetilde{T}_X \otimes \widetilde{T}_X)$
is one-dimensional and is generated by the natural morphism
$\widetilde{T}_X(-1) = \widetilde{T}_X^\vee \to \cO \to \widetilde{T}_X$ provided by
\eqref{eq:exact-sequence-extension-tangent-bundle} and \eqref{eq:dual-T_X_tilde}.
Thus, perhaps up to rescaling, the square \eqref{eq:commutative-square} is commutative.

Taking the total complex of \eqref{eq:commutative-square} we get
\begin{equation*}
  \widetilde{T}_X(-1) \xrightarrow{(\alpha^t, - \beta^t)} \left( V^{\omega_1} \oplus \bC \right) \otimes \cO \xrightarrow{(\alpha, \beta)} \widetilde{T}_X
\end{equation*}
Since $(\alpha, \beta)$ is surjective, the map $(\alpha^t, - \beta^t)$ is injective. Thus, we have a monad. Taking the alternating sum of ranks of its terms we see that its cohomology is $\rG$-equivaraint vector bundle of rank $21$.
Since we have $V^{\omega_1} = \mathfrak{g}$, the weights of $V^{\omega_1}$ are the roots \eqref{eq:roots-F4}.
Using~\eqref{eq:fundamental-weights-F4} we see that the root $(1,-1,0,0) = 2\omega_4 - \omega_1$ is $\rL$-dominant,
and since the bundle $\cU^{2\omega_4}(-1)$ is of the required rank $21$, the cohomology of the monad has to be $\cU^{2\omega_4}(-1)$.
\end{proof}

Applying Lemma \ref{lemma:tensor-products-F4-P1} and Corollaries \ref{corollary:duals} and \ref{corollary:acyclicities}
to compute the necessary $\Ext$-groups we obtain the inclusion
\begin{equation}\label{eq:semiorthogonality-omega_2(-1)}
  \cU^{\omega_2}(-1) \in \big \langle \cU^{\omega_4}(1), \cO(2), \cU^{\omega_4}(2), \dots, \cO(7), \cU^{\omega_4}(7) \big \rangle^\perp.
\end{equation}
Therefore, the left mutation of $\cU^{\omega_2}(-1)$ with respect to
$\langle \cU^{\omega_4}(1), \cO(2), \cU^{\omega_4}(2), \dots, \cO(7), \cU^{\omega_4}(7) \big \rangle$
is $\cU^{\omega_2}(-1)$ itself. Combining it with Lemma \ref{lemma:definition-of-T_X_tilde}(3) we get
\begin{equation*}
  \mathbf{L}_{\langle \cO(1), \cU^{\omega_4}(1), \dots, \cO(7), \cU^{\omega_4}(7) \rangle}(\cU^{\omega_2}(-1)) = T_X,
\end{equation*}
and, therefore, we have
\begin{equation}\label{eq:semiorthogonality-T_X}
  T_X \in \big \langle \cO(1), \cU^{\omega_4}(1), \dots, \cO(7), \cU^{\omega_4}(7) \big \rangle^\perp.
\end{equation}
Further, by Lemma \ref{lemma:definition-of-T_X_tilde}(3) we have $\mathbf{R}_{\langle  \cO \rangle}(T_X) = \widetilde{T}_X$
and by Lemma \ref{lemma:collection-O-and-U} we have
\begin{equation*}
  \cO \in \big \langle \cO(1), \cU^{\omega_4}(1), \dots, \cO(7), \cU^{\omega_4}(7) \big \rangle^\perp.
\end{equation*}
Hence, we obtain
\begin{equation}\label{eq:semiorthogonality-T_X_tilde}
  \widetilde{T}_X \in {}^\perp \big \langle \cO \big \rangle  \cap  \big \langle \cO(1), \cU^{\omega_4}(1), \dots, \cO(7), \cU^{\omega_4}(7) \big \rangle^\perp.
\end{equation}

\begin{lemma}\label{lemma:collection-Ttilde}
  The sequence of vector bundles $\widetilde{T}_X, \widetilde{T}_X(1), \dots, \widetilde{T}_X(7)$ is exceptional.
\end{lemma}

\begin{proof}
  It is enough to prove
  \begin{equation}\label{eq:proof-exceptionality-Ttilde-1}
    \Ext^\bullet(\widetilde{T}_X, \widetilde{T}_X) = \bC[0]
  \end{equation}
  and
  \begin{equation}\label{eq:proof-exceptionality-Ttilde-2}
    \Ext^\bullet(\widetilde{T}_X(t), \widetilde{T}_X) = 0 \quad \text{for} \quad 1 \leq t \leq 7.
  \end{equation}
  We split the proof into 3 cases.

  \medskip

  \noindent \emph{Identity \eqref{eq:proof-exceptionality-Ttilde-2} for $2 \leq t \leq 6$.}
  Exact sequences \eqref{eq:exact-sequence-tangent-bundle}--\eqref{eq:exact-sequence-extension-tangent-bundle}
  imply that we have an inclusion
  \begin{equation*}
    \widetilde{T}_X \in \cC \coloneqq \big \langle \cO, \cO(1), \cU^{\omega_2}(-1) \big \rangle
  \end{equation*}
  and, therefore, it is enough to show that we have
  \begin{equation*}
    \Hom(\cC(t), \cC) = 0 \quad \text{for} \quad 2 \leq t \leq 6,
  \end{equation*}
  which in its turn follows from
  \begin{equation*}
    \begin{aligned}
      & \Ext^\bullet(\cO(t), \cO) = H^\bullet(X, \cO(-t)) = 0  \quad \text{for} \quad 1 \leq t \leq 7, \\
      & \Ext^\bullet(\cU^{\omega_2}(t-1), \cO) = H^\bullet(X, \cU^{\omega_2}(-t-2)) = 0  \quad \text{for} \quad 1 \leq t \leq 6, \\
      & \Ext^\bullet(\cO(t), \cU^{\omega_2}(-1)) = H^\bullet(X, \cU^{\omega_2}(-t-1)) = 0  \quad \text{for} \quad 2 \leq t \leq 7, \\
      & \Ext^\bullet(\cU^{\omega_2}(t-1), \cU^{\omega_2}(-1)) = 0 \quad \text{for} \quad 2 \leq t \leq 6.
    \end{aligned}
  \end{equation*}
  The first three vanishings hold immediately by Corollary \ref{corollary:acyclicities}.
  To show the last one, using Lemma \ref{lemma:tensor-products-F4-P1} and Corollary \ref{corollary:duals}, we rewrite
  \begin{equation*}
    \Ext^\bullet(\cU^{\omega_2}(t-1), \cU^{\omega_2}(-1)) = H^\bullet(X, \cU^{2\omega_2}(-3-t) \oplus \cU^{2\omega_3}(-2-t) \oplus \cU^{2\omega_4}(-1-t) \oplus \cO(-t))
  \end{equation*}
  and again apply Corollary \ref{corollary:acyclicities}.

  \medskip

  \noindent \emph{Identity \eqref{eq:proof-exceptionality-Ttilde-2} for $t = 1$ and $t = 7$.} Note that the case $t = 7$ is equivalent to $t = 1$ by Serre duality. Thus, we only need to consider the case $t = 1$. By Lemma \ref{lemma:monad-U} it is enough to show
  \begin{equation*}
    \Ext^\bullet(\widetilde{T}_X(1), \cO) = 0 , \quad \Ext^\bullet(\widetilde{T}_X(1), \widetilde{T}_X(-1)) = 0 , \quad \Ext^\bullet(\widetilde{T}_X(1), \cU^{2\omega_4}(-1)) = 0.
  \end{equation*}
  The first one and the last one follow from Lemma \ref{lemma:tensor-products-F4-P1} and Corollaries
  \ref{corollary:duals} and \ref{corollary:acyclicities} by replacing $\widetilde{T}_X(1)$ with its
  semisimplification \eqref{eq:semisimplification-T_X_tilde}. The second one follows
  from \eqref{eq:proof-exceptionality-Ttilde-2} with $t = 2$, which we already know.

  \medskip

  \noindent \emph{Identity \eqref{eq:proof-exceptionality-Ttilde-1}.} By Lemma \ref{lemma:monad-U} it is enough to prove
  \begin{equation*}
    \Ext^\bullet(\widetilde{T}_X, \cO) = 0 , \quad \Ext^\bullet(\widetilde{T}_X, \widetilde{T}_X(-1)) = 0 , \quad \Ext^\bullet(\widetilde{T}_X, \cU^{2\omega_4}(-1)) = \Bbbk[-1].
  \end{equation*}
  The first one holds by \eqref{eq:semiorthogonality-T_X_tilde}. The second one follows
  from \eqref{eq:proof-exceptionality-Ttilde-2} with $t = 1$, which we already know.
  To see the last one, according to \eqref{eq:semisimplification-T_X_tilde}, it is enough to show
  \begin{equation*}
  \begin{aligned}
    & \Ext^\bullet(\cO(t), \cU^{2\omega_4}(-1)) = 0, \quad \text{for} \quad 0 \leq t \leq 1, \\
    & \Ext^\bullet(\cU^{\omega_2}(-1)), \cU^{2\omega_4}(-1)) = \Bbbk[-1],
  \end{aligned}
  \end{equation*}
  which follow easily from Lemma \ref{lemma:tensor-products-F4-P1}, Corollaries
  \ref{corollary:duals} and \ref{corollary:acyclicities}, and \eqref{eq:cohomology-omega_2(-2)}.
\end{proof}

\begin{lemma}
  The collection \eqref{eq:collection-F4-P1} is exceptional.
\end{lemma}

\begin{proof}
  In view of the previous lemmas it is enough to show
  \begin{equation*}
    \Ext^\bullet(\widetilde{T}_X(t), \cO) = 0 \quad \text{for} \quad 0 \leq t \leq 7
  \end{equation*}
  and
  \begin{equation*}
    \Ext^\bullet(\widetilde{T}_X(t), \cU^{\omega_4}) = 0 \quad \text{for} \quad 0 \leq t \leq 7.
  \end{equation*}
  We treat each of these cases separately.
  \begin{enumerate}
    \item For $t = 0$ this follows from \eqref{eq:semiorthogonality-T_X_tilde}. For $1 \leq t \leq 6$
    this follows from Corollaries \ref{corollary:duals} and \ref{corollary:acyclicities} by replacing
    $\widetilde{T}_X$ with its semisimplification \eqref{eq:semisimplification-T_X_tilde}. For $t = 7$
    this follows from the case $t = 0$ by Serre duality, dualization of factors, and \eqref{eq:dual-T_X_tilde}.

    \item For $1 \leq t \leq 7$ this follows from Lemma \ref{lemma:tensor-products-F4-P1} and
    Corollaries \ref{corollary:duals}, \ref{corollary:acyclicities} by replacing $\widetilde{T}_X$
    with its semisimplification \eqref{eq:semisimplification-T_X_tilde}.

    Let us now assume $t = 0$. First we note that dualizing each factor and using Corollary \ref{corollary:duals} and \eqref{eq:dual-T_X_tilde} we have
    \begin{equation*}
      \Ext^\bullet(\widetilde{T}_X, \cU^{\omega_4}) \cong \Ext^\bullet(\cU^{\omega_4}, \widetilde{T}_X).
    \end{equation*}
    Thus, by Lemma \ref{lemma:monad-U} it is enough to prove
    \begin{equation*}
      \Ext^\bullet(\cU^{\omega_4}, \cO) = 0 , \quad \Ext^\bullet(\cU^{\omega_4}, \widetilde{T}_X(-1)) = 0 , \quad \Ext^\bullet(\cU^{\omega_4}, \cU^{2\omega_4}(-1)) = 0.
    \end{equation*}
    The first one follows from Lemma \ref{lemma:collection-O-and-U}. The second one
    follows from the case $t = 1$ we have just proved. Finally, the last one follows
    from Lemma \ref{lemma:tensor-products-F4-P1} and Corollaries \ref{corollary:duals}
    and \ref{corollary:acyclicities}.
    \end{enumerate}
\end{proof}

\begin{remark}
  We have seen in the proof of the above lemma that the bundles $\cU^{\omega_4}$ and $\widetilde{T}_X$ are completely
  orthogonal. This also hold for their twists, i.e. we have
  \begin{equation*}
    \Ext^\bullet(\widetilde{T}_X(t), \cU^{\omega_4}(t)) \cong \Ext^\bullet(\cU^{\omega_4}(t), \widetilde{T}_X(t)) = 0.
  \end{equation*}
\end{remark}

\section{Proof of fullness}
\label{section:fullness}

Here we prove that the collection \eqref{eq:collection-F4-P1} is full. The proof goes via restriction to some closed subvarieties of $X$, where a full exceptional collection is already known. In this respect the proof is similar to the proofs of fullness given in \cite{KS21, Ku08a}.

Let $\cD$ be the full triangulated subcategory of $\Db(X)$ generated by the exceptional
collection \eqref{eq:collection-F4-P1}, i.e. we have
\begin{equation}\label{eq:subcategory-D}
  \cD = \big\langle \cO, \cU^{\omega_4}, \widetilde{T}_X, \cO(1), \cU^{\omega_4}(1), \widetilde{T}_X(1), \dots,
  \cO(7), \cU^{\omega_4}(7), \widetilde{T}_X(7) \big\rangle \subset \Db(X).
\end{equation}
Since $\cD$ is generated by an exceptional collection, it is admissible. Therefore, there exists a semiorthogonal decomposition
\begin{equation*}
  \Db(X) = \langle \cD^\perp, \cD \rangle.
\end{equation*}
To prove fullness of \eqref{eq:collection-F4-P1} we need to prove $\cD^\perp = 0$.

\smallskip

This section is structured as follows. In Section~\ref{subsection:some-geometry} we
begin with some geometric facts that eventually allow us to cover $X$ by copies
of the even orthogonal Grassmannian $\OG(2,8)$ (see Lemma~\ref{lemma:zero-loci}).
These facts are also interesting in their own right. In Section~\ref{subsection:auxiliary-complexes} we collect some
auxiliary complexes that allow us to conclude that certain additional $\rG$-equivariant
vector bundles are already contained in the subcategory $\cD$ (see Corollary~\ref{corollary:containments}).
Finally, in Section~\ref{subsection:proof-of-fullness} we prove fullness by showing
in Proposition~\ref{proposition:fullness} that $\cD^\perp = 0$.

\subsection{Some geometry}
\label{subsection:some-geometry}

Since $X$ is the adjoint variety of a simple group $\rG$, there is a closed embedding
\begin{equation*}
  X \subset \bP(\mathfrak{g}),
\end{equation*}
identifying $X$ with the orbit of the highest root $[\Theta] \in \bP(\mathfrak{g})$,
where $\rG$ acts on its Lie algebra $\mathfrak{g}$ via the adjoint action.
Moreover, since we have an isomorphism of $\rG$-representations $\mathfrak{g} \cong V^{\omega_1}$, the above embedding becomes
\begin{equation*}
  X \subset \bP(V^{\omega_1}).
\end{equation*}

Recall from \cite[p.202--213]{Bourbaki} that we have the following formulas for roots of the root systems $\rF_4$, $\rB_4$ and $\rD_4$
\begin{equation*}
  \begin{aligned}
    & \rF_4 \colon \quad \pm e_i, \\
    & \rB_4 \colon \quad \pm e_i, \\
    & \rD_4 \colon
  \end{aligned}
  \quad
  \begin{aligned}
    & \pm e_i \pm e_j, \\
    & \pm e_i \pm e_j \\
    & \pm e_i \pm e_j
  \end{aligned}
  \quad
  \begin{aligned}
    & \frac{1}{2}(\pm e_1 \pm e_2 \pm e_3 \pm e_4) \\
    & \\
    & \\
  \end{aligned}
\end{equation*}
The roots for $\rF_4$ already appeared in \eqref{eq:roots-F4}.
From these formulas we clearly see the inclusions of root systems $\rD_4 \subset \rB_4 \subset \rF_4$,
that give rise to embeddings of the corresponding connected simply connected simple algebraic groups
of the corresponding Dynkin types
\begin{equation*}
  \bar{\rG} \subset \widetilde{\rG} \subset \rG.
\end{equation*}
Note that for all three groups we have a common maximal torus
\begin{equation}\label{eq:common-maximal-torus}
  \rT \subset \bar{\rG} \subset \widetilde{\rG} \subset \rG.
\end{equation}
We also have embeddings of homogeneous spaces
\begin{equation*}
  \bar{\rG}/\bar{\rP} \subset \widetilde{\rG}/\widetilde{\rP} \subset \rG/\rP,
\end{equation*}
where
\begin{equation*}
  \widetilde{\rP} = \rP \cap \widetilde{\rG} \quad \text{and} \quad \bar{\rP} = \rP \cap \bar{\rG}.
\end{equation*}
Since the highest root of $\rF_4$ is $\Theta = e_1 + e_2 = \omega_1$ and since its restriction
to $\rB_4$ (resp. $\rD_4$) is the second fundamental weight of $\rB_4$ (resp. $\rD_4$),
we have
\begin{equation*}
  \widetilde{\rP} = \widetilde{\rP}_2 \quad \text{and} \quad \bar{\rP} = \bar{\rP}_2,
\end{equation*}
where we use the Bourbaki labelling of vertices in Dynkin diagrams $\rB_4$ and $\rD_4$
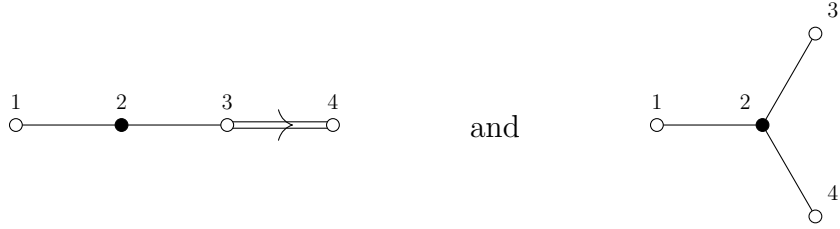
\begin{equation*}
  \dynkin[edge length = 2em,labels*={1,...,4}, arrow width=2mm, scale=1.7]B{o*oo}
  \qquad \qquad \text{and} \qquad \qquad
  \dynkin[edge length = 2em,labels*={1,...,4}, scale=1.7]D{o*oo}
\end{equation*}
and the nodes depicted in solid black correspond to the parabolics $\widetilde{\rP}$ and $\bar{\rP}$ respectively.

For the remainder of this subsection we denote by $V_{\widetilde{\rG}}^\lambda$ (resp. $V_{\bar{\rG}}^\lambda$)
the irreducible representation of the group $\widetilde{\rG}$ (resp. $\bar{\rG}$) with the highest
weight $\lambda$. Similarly, we denote by $\widetilde{\cU}^\lambda$ (resp. $\bar{\cU}^\lambda$)
the $\widetilde{\rG}$-equivariant (resp. $\bar{\rG}$-equivariant) vector bundle on
$\widetilde{\rG}/\widetilde{\rP}$ (resp. $\bar{\rG}/\bar{\rP}$) defined by the irreducible representation
with the highest weight $\lambda$ of the Levi subgroup $\widetilde{\rL}$ (resp. $\bar{\rL}$) of the
parabolic $\widetilde{\rP}$ (resp. $\bar{\rP}$).

A simple computation with root systems (e.g. using \cite[p.202--213]{Bourbaki}) shows that the inclusions $\bar{\rG} \subset \widetilde{\rG} \subset \rG$ induce the following maps on the fundamental weights
\begin{equation}\label{eq:D4-B4-F4-maps-on-weights}
\begin{aligned}
  & \omega_1 \mapsto \tildeomega_2 \\
  & \omega_2 \mapsto \tildeomega_1 + \tildeomega_3 \\
  & \omega_3 \mapsto \tildeomega_1 + \tildeomega_4 \\
  & \omega_4 \mapsto \tildeomega_1
\end{aligned}
\qquad , \qquad
\begin{aligned}
  & \omega_1 \mapsto \baromega_2 \\
  & \omega_2 \mapsto \baromega_1 + \baromega_3 + \baromega_4 \\
  & \omega_3 \mapsto \baromega_1 + \baromega_4 \\
  & \omega_4 \mapsto \baromega_1
\end{aligned}
\qquad \text{and} \qquad
\begin{aligned}
  & \tildeomega_1 \mapsto \baromega_1 \\
  & \tildeomega_2 \mapsto \baromega_2 \\
  & \tildeomega_3 \mapsto \baromega_3 + \baromega_4 \\
  & \tildeomega_4 \mapsto \baromega_4
\end{aligned}
\qquad ,
\end{equation}
where we denote by $\tildeomega_i$ (resp. $\baromega_i$) the fundamental weights of $\widetilde{\rG}$ (resp. $\bar{\rG}$).

Let $V^{\lambda}$ be an irreducible highest weight representation of $\rG$. Then its restriction
$V^{\lambda}|_{\widetilde{\rG}}$ decomposes into a direct sum of irreducible representations of
$\widetilde{\rG}$. Moreover, the image of $\lambda$ under \eqref{eq:D4-B4-F4-maps-on-weights}
is the highest weight of one of the summands. The same of course holds for restrictions from
$\rG$ to $\bar{\rG}$ and from $\widetilde{\rG}$ to $\bar{\rG}$.
For example, this way immediately obtain inclusions
\begin{equation}\label{eq:inclusions-restricted-representations}
  V_{\bar{\rG}}^{\baromega_2} \subset V_{\widetilde{\rG}}^{\tildeomega_2} \subset V^{\omega_1}.
\end{equation}

\medskip

We now give a precise description of the restrictions of $V^{\omega_1}$ and $V^{\omega_4}$.
\begin{lemma}\label{lemma:restrictions-of-representations}
  \
  \begin{enumerate}
    \item The normalizer $N_{\rG}(\bar{\rG})$ is a semidirect product of $\bar{\rG}$ and
    the symmetric group $\mathrm{S}_3$, where the group $\mathrm{S}_3$ acts on $\bar{\rG}$ by
    the outer automorphisms that correspond to permutations of the nodes $1,3,4$ of
    the Dynkin diagram of type $\rD_4$.

    \smallskip

    \item Let $V$ be a finite dimensional representation of $\rG$ and $V|_{\bar{\rG}}$
    its restriction to $\bar{\rG}$.
    If $W \subset V|_{\bar{\rG}}$ is a $\bar{\rG}$-subrepresentation, then for any
    $g \in N_{\rG}(\bar{\rG})$ the subspace $g(W) \subset V|_{\bar{\rG}}$ is also
    a $\bar{\rG}$-subrepresentation. Moreover, $W$ and $g(W)$ differ by the automorphism
    of $\bar{\rG}$ given by the conjugation with $g$.

    \smallskip

    \item If $V^{\lambda}|_{\bar{\rG}}$ has $V_{\bar{\rG}}^{\baromega_i}$ as a direct
    summand for some $i \in \{1,3,4\}$, then it contains $V_{\bar{\rG}}^{\baromega_i}$
    for all $i \in \{1,3,4\}$ as direct summands.

    \smallskip

    \item
    We have
    \begin{equation*}
      \begin{aligned}
        V^{\omega_1}|_{\widetilde{\rG}} & = V_{\widetilde{\rG}}^{\tildeomega_2} \oplus V_{\widetilde{\rG}}^{\tildeomega_4} \\
        V^{\omega_4}|_{\widetilde{\rG}} & = V_{\widetilde{\rG}}^{\tildeomega_1} \oplus V_{\widetilde{\rG}}^{\tildeomega_4} \oplus \Bbbk
      \end{aligned}
      \hspace{50pt}
      \begin{aligned}
        V^{\omega_1}|_{\bar{\rG}} & = V_{\bar{\rG}}^{\baromega_2}
        \oplus V_{\bar{\rG}}^{\baromega_1} \oplus V_{\bar{\rG}}^{\baromega_3} \oplus V_{\bar{\rG}}^{\baromega_4} \\
        V^{\omega_4}|_{\bar{\rG}} & = V_{\bar{\rG}}^{\baromega_1} \oplus V_{\bar{\rG}}^{\baromega_3} \oplus V_{\bar{\rG}}^{\baromega_4} \oplus \Bbbk^2
      \end{aligned}
    \end{equation*}
  \end{enumerate}
\end{lemma}

\begin{proof}
  (1) Since all maximal tori in semisimple groups are conjugate, every coset
  in $N_{\rG}(\bar{\rG})/\bar{\rG}$ has a representative from $N_{\rG}(\rT)$.
  Moreover, every two such representatives differ by an element of $N_{\bar{\rG}}(\rT)$.
  Therefore, we obtain isomorphisms of groups
  \begin{equation*}
    N_{\rG}(\bar{\rG})/\bar{\rG} \cong N_{\rG}(\rT) / N_{\bar{\rG}}(\rT) \cong \rW_{\rG}/\rW_{\bar{\rG}}.
  \end{equation*}
  By \cite[p.~213]{Bourbaki} we know that $\rW_{\rG}/\rW_{\bar{\rG}} \cong \mathrm{S}_3$,
  and that $\rW_{\rG}$ is a semidirect product of $\rW_{\bar{\rG}}$ and the symmetric group
  $\mathrm{S}_3$, where $\mathrm{S}_3$ acts on $\rW_{\bar{\rG}}$ by the outer automorphisms
  that correspond to permutations of the nodes $1,3,4$ of the Dynkin diagram of type $\rD_4$.
  This imlies the claim.

  \medskip

  (2) This follows immediately from definitions.

  \medskip

  (3) This follows from Parts (1) and (2). Indeed, let $g \in \rG$ be an element
  of $\mathrm{S}_3 \subset N_{\rG}(\bar{\rG})$ as in Part (1). Then by Part (2)
  we have $g(V_{\bar{\rG}}^{\baromega_i}) = V_{\bar{\rG}}^{\sigma(\baromega_{i})}$,
  where $\sigma$ is the permutation of the set $\{ \baromega_1, \baromega_3, \baromega_4 \}$
  induced by $g$. This way we can obtain $V_{\bar{\rG}}^{\baromega_i}$ for all $i \in \{1,3,4\}$.

  \medskip

  (4) Let us consider the restriction $V^{\omega_1}|_{\widetilde{\rG}}$.
  Note that under \eqref{eq:D4-B4-F4-maps-on-weights} the weight $\omega_1$ maps
  to $\tildeomega_2$. Hence, we have
  \begin{equation*}
    V^{\omega_1}|_{\widetilde{\rG}} = V_{\widetilde{\rG}}^{\tildeomega_2} \oplus W \, ,
  \end{equation*}
  where $W$ is a representation of $\widetilde{\rG}$ of dimension $\dim W = 16$,
  as we have $\dim V^{\omega_1} = 52$ and $\dim V_{\widetilde{\rG}}^{\tildeomega_2} = 36$.
  Since we have $\rT \subset \widetilde{\rG} \subset \rG$, the weights of $W$ with
  respect to $\rT$ are a subset of the weights of $V^{\omega_1}$ with respect to $\rT$.
  As the multiplicity of the zero weight of a simple Lie algebra is equal to its rank,
  we obtain that the multiplicities of the zero weight in $V^{\omega_1} = \mathfrak{g}$
  and in $V_{\widetilde{\rG}}^{\tildeomega_2} = \widetilde{\mathfrak{g}}$ are both
  equal to $4$. This implies that the multiplicity of the zero weight in $W$ must
  be zero.
  Now we note that the only irreducible representations of $\widetilde{\rG}$ of
  dimension smaller or equal to $16$ are
  $\Bbbk, V_{\widetilde{\rG}}^{\tildeomega_1}, V_{\widetilde{\rG}}^{\tildeomega_4}$,
  whose dimensions are $1,9,16$ respectively. Since $W$ has no zero weight space,
  its decomposition into irreducibles cannot involve the trivial representation $\Bbbk$.
  Hence, for dimension reasons we obtain $W = V_{\widetilde{\rG}}^{\tildeomega_4}$.

  Similarly one can prove the claim for $V^{\omega_4}|_{\widetilde{\rG}}$
  using that the multiplicity of the zero weight in $V^{\omega_4}$ is $2$.

  To obtain restrictions $V^{\omega_1}|_{\bar{\rG}}$ and $V^{\omega_4}|_{\bar{\rG}}$
  we use the restrictions $V^{\omega_1}|_{\widetilde{\rG}}$ and $V^{\omega_4}|_{\widetilde{\rG}}$,
  \eqref{eq:D4-B4-F4-maps-on-weights} and Part (3) of this lemma.
\end{proof}

\begin{lemma}\label{lemma:B4-and-D4-as-linear-sections}
  The subvarieties $\widetilde{\rG}/\widetilde{\rP} \subset \rG/\rP$ and $\bar{\rG}/\bar{\rP} \subset \widetilde{\rG}/\widetilde{\rP}$ are linear sections in their natural projective embeddings $\rG/\rP \subset \bP(V^{\omega_1})$
  and $\widetilde{\rG}/\widetilde{\rP} \subset \bP(V_{\widetilde{\rG}}^{\tildeomega_2})$.
  More precisely, we have
  \begin{equation*}
    \widetilde{\rG}/\widetilde{\rP} = \rG/\rP \underset{\bP(V^{\omega_1})}{\times} \bP(V_{\widetilde{\rG}}^{\tildeomega_2})
    \hspace{30pt} \text{and} \hspace{30pt}
    \bar{\rG}/\bar{\rP} = \widetilde{\rG}/\widetilde{\rP} \underset{\bP(V_{\widetilde{\rG}}^{\tildeomega_2})}{\times} \bP(V_{\bar{\rG}}^{\baromega_2}),
  \end{equation*}
  where the inclusions $V_{\widetilde{\rG}}^{\tildeomega_2} \subset V^{\omega_1}$ and
  $V_{\bar{\rG}}^{\baromega_2} \subset V_{\widetilde{\rG}}^{\tildeomega_2}$ are given
  by \eqref{eq:inclusions-restricted-representations}.
\end{lemma}

\begin{proof}
  Let us first consider the case $\widetilde{\rG}/\widetilde{\rP} \subset \rG/\rP$. From the classification
  of symmetric spaces (see for example \cite[p.~152]{Timashev}) we see that $\rG/\widetilde{\rG}$ is a
  symmetric space. Hence, $\widetilde{\rG}$ is a fixed-point subgroup of an involution on $\rG$. Its differential induces a direct sum decomposition of the Lie algebra of $\rG$ into eigenspaces
  \begin{equation}\label{eq:eigenspace-decomposition}
    \mathfrak{g} = \mathfrak{g}_{0} \oplus \mathfrak{g}_{1} \quad \text{with} \quad \mathfrak{g}_{0} = \Lie(\widetilde{\rG}).
  \end{equation}

  Since the Dynkin diagram of type $\rF_4$ has no automorphisms,
  the group $\rG$ has no outer automorphisms (see \cite[Theorem~27.4 and (A.8)]{Humphreys}).
  Hence, this involution is a conjugation by an element $\tau \in \rG$ of order $2$.
  Moreover, since the involution fixes $\widetilde{\rG}$, it also fixes the maximal
  torus $\rT \subset \widetilde{\rG} \subset \rG$ (see \eqref{eq:common-maximal-torus}).
  Since $\centralizer_{\rG}(\rT) = \rT$, we obtain $\tau \in \rT$.

  Since $\rT \subset \rP$, the multiplications by $\tau$ and $\tau^{-1}$ preserve
  the parabolic $\rP$ and, therefore, the induced action on $\rG/\rP$
  is given by left multiplication with $\tau$. Applying \cite[Theorem A]{Richardson}
  we see that $\widetilde{\rG}/\widetilde{\rP} \subset \rG/\rP$ is a connected component
  of the fixed locus $\left( \rG/\rP \right)^\tau$.

  The action of the maximal torus $\rT$ on $\rG/\rP$ preserves $\left( \rG/\rP \right)^\tau$,
  since multiplication by $\tau$ commutes with multiplication by any element of $\rT$. Hence,
  we have inclusions
  \begin{equation*}
    \widetilde{\rG}/\widetilde{\rP} \subset \left( \rG/\rP \right)^\tau \subset \rG/\rP
  \end{equation*}
  and the maximal torus $\rT$ acts by left multiplication on these spaces in a compatible way.
  Hence, we also have inclusions on the sets of $\rT$-fixed points
  \begin{equation*}
    \left( \widetilde{\rG}/\widetilde{\rP} \right)^{\rT} \subset \Big( \left( \rG/\rP \right)^\tau  \Big)^{\rT} \subset \Big( \rG/\rP  \Big)^{\rT}.
  \end{equation*}
  The cardinalities of $\left( \widetilde{\rG}/\widetilde{\rP} \right)^{\rT}$ and
  $\Big( \rG/\rP  \Big)^{\rT}$ are equal to the number of Schubert cells in the respective
  varieties and we easily compute
  \begin{equation*}
    \frac{|\rW_{\rG}|}{|\rW_{\rL}|} = \frac{1152}{48} = 24 \quad \text{and} \quad
    \frac{|\rW_{\widetilde{\rG}}|}{|\rW_{\widetilde{\rL}}|} = \frac{384}{2 \cdot 16} = 24,
  \end{equation*}
  where the sizes of the Weyl groups can be found in \cite{Bourbaki}, for example.

  Thus, we have equalities
  \begin{equation*}
    \left( \widetilde{\rG}/\widetilde{\rP} \right)^{\rT} = \Big( \left( \rG/\rP \right)^\tau  \Big)^{\rT} = \Big( \rG/\rP  \Big)^{\rT}.
  \end{equation*}
  Therefore, since we already know that $\widetilde{\rG}/\widetilde{\rP}$ is a connected component of $\left( \rG/\rP \right)^\tau$,
  we obtain that $\left( \rG/\rP \right)^\tau$ is connected and we have the equality
  \begin{equation*}
    \widetilde{\rG}/\widetilde{\rP} = \left( \rG/\rP \right)^\tau.
  \end{equation*}

  Consider now the $\rG$-equivariant embedding
  \begin{equation*}
    \rG/\rP \subset \bP(\mathfrak{g})
  \end{equation*}
  and note that we have the equality
  \begin{equation*}
    \left( \rG/\rP \right)^\tau = \rG/\rP \underset{\bP(\mathfrak{g})}{\times} \bP(\mathfrak{g})^\tau
  \end{equation*}
  as a scheme-theoretic intersection. Using \eqref{eq:eigenspace-decomposition} we see that there is a disjoint union decomposition
  \begin{equation*}
    \bP(\mathfrak{g})^\tau = \bP(\mathfrak{g}_{0}) \sqcup \bP(\mathfrak{g}_{1}).
  \end{equation*}
  Hence, we obtain
  \begin{equation*}
    \widetilde{\rG}/\widetilde{\rP} = \rG/\rP \underset{\bP(\mathfrak{g})}{\times} \bP(\mathfrak{g}_0)
    \quad \text{and} \quad
    \rG/\rP \underset{\bP(\mathfrak{g})}{\times} \bP(\mathfrak{g}_1) = \emptyset.
  \end{equation*}
  Since we have $\mathfrak{g}_0 = V_{\widetilde{\rG}}^{\tildeomega_2}$, the first claim of the lemma is proved.

  \medskip

  A very similar argument allows to prove the second claim, since $\widetilde{\rG}/\bar{\rG} \cong \SO_9/\SO_8$ is again a symmetric space (see \cite[p.~152]{Timashev}). We omit the details.
\end{proof}

\begin{corollary}\label{corollary:D4-as-linear-sections}
  We have
  \begin{equation*}
    \bar{\rG}/\bar{\rP} = \rG/\rP \underset{\bP(V^{\omega_1})}{\times} \bP(V_{\bar{\rG}}^{\baromega_2}),
  \end{equation*}
  where the inclusion $V_{\bar{\rG}}^{\baromega_2} \subset V^{\omega_1}$ is given
  by \eqref{eq:inclusions-restricted-representations}.
\end{corollary}

Our next goal is to prove that $\bar{\rG}/\bar{\rP} \subset \rG/\rP$ is the zero locus of a regular section of $\cU^{\omega_4}$.

\begin{lemma}\label{lemma:zero-loci}
  \
  \begin{enumerate}
    \item[(i)] Up to a scalar multiple there exists a unique non-zero global section $s_0 \in V^{\omega_4}$ of the vector bundle $\cU^{\omega_4}$ such that
    \begin{equation*}
      Z(s_0) = \widetilde{\rG}/\widetilde{\rP}.
    \end{equation*}

    \item[(ii)] There exist a one parameter family of global sections $s_t \in V^{\omega_4}$ of the vector bundle $\cU^{\omega_4}$ such that
    \begin{equation*}
      Z(s_t) = \bar{\rG}/\bar{\rP}.
    \end{equation*}
    In particular, these global sections are regular.
  \end{enumerate}
\end{lemma}

\begin{proof}
  To show (i) we proceed as follows. By Lemma \ref{lemma:restrictions-of-representations}(4) we know
  that $V^{\omega_4}|_{\widetilde{\rG}}$ has a one-dimensional trivial representation as a direct
  summand. Let $s_0 \in V^{\omega_4}$ be any non-zero element in this subspace. Our goal is to
  show $Z(s_0) = \widetilde{\rG}/\widetilde{\rP}$. Consider the exact sequence
  \begin{equation*}
    \left( \cU^{\omega_4} \right)^\vee \xrightarrow{s_0^\vee} \cO_{\rG/\rP} \to \cO_{Z(s_0)} \to 0.
  \end{equation*}
  The image of $s_0^\vee$ is the ideal sheaf $I_{Z(s_0)}$ of $Z(s_0)$. By Corollary \ref{corollary:duals} we have
  $\left( \cU^{\omega_4} \right)^\vee(1) = \cU^{\omega_4}$ and, therefore, both $\left( \cU^{\omega_4} \right)^\vee(1)$
  and $I_{Z(s_0)}(1)$ are globally generated. Hence, $Z(s_0)$ inside $\rG/\rP$ is a linear section with the subspace
  $H^0(\rG/\rP, I_{Z(s_0)}(1))^\perp$. Since $s_0$ is fixed by $\widetilde{\rG}$, the map
  \begin{equation*}
    H^0(\rG/\rP, \cU^{\omega_4}) = V^{\omega_4} \xrightarrow{s_0^\vee(1)} H^0(\rG/\rP, \cO(1)) = V^{\omega_1},
  \end{equation*}
  is $\widetilde{\rG}$-equivariant and non-zero. Hence, by Lemma \ref{lemma:restrictions-of-representations}(4) its image
  is
  $V^{\tildeomega_4}_{\widetilde{\rG}}$ and we get the desired
  \begin{equation*}
    H^0(\rG/\rP, I_{Z(s_0)}(1))^\perp = V^{\tildeomega_2}_{\widetilde{\rG}}.
  \end{equation*}
  Finally, applying Lemma \ref{lemma:B4-and-D4-as-linear-sections} we get the claim.

  \medskip

  The proof of (ii) is similar but the argument becomes a bit more subtle.
  By Lemma \ref{lemma:restrictions-of-representations}(4) we know
  that $V^{\omega_4}|_{\bar{\rG}}$ has a two-dimensional trivial representation
  as a direct summand.
  Let us denote this subspace $H \subset V^{\omega_4}$; the section $s_0$ considered
  in part (i) is contained in $H$.
  As in the proof of (i), any non-zero element $s \in H$ defines a non-zero
  $\bar{\rG}$-equivariant map
  \begin{equation}\label{eq:lemma-zero-locus-proof}
    H^0(\rG/\rP, \cU^{\omega_4}) = V^{\omega_4} \xrightarrow{s^\vee(1)} H^0(\rG/\rP, \cO(1)) = V^{\omega_1},
  \end{equation}
  and to prove the claim of the lemma, we need to determine the dimension of
  its image for variying sections $s$. Since by Lemma \ref{lemma:restrictions-of-representations}(4)
  the representations $V^{\omega_4}|_{\bar{\rG}}$ and $V^{\omega_1}|_{\bar{\rG}}$
  have three irreducible direct summands in common, we have to be more careful
  than in part (i).

  Let us vary the section $s \in V^{\omega_4}$, i.e. we consider the canonical
  evaluation morphism $V^{\omega_4} \otimes \cO \to \cU^{\omega_4}$, take its dual
  to get $\cU^{\omega_4}(-1) \to {(V^{\omega_4})}^\vee \otimes \cO$,
  twist the result by $\cO(1)$ and finally take global sections. This way we obtain
  a $\rG$-equivariant map $V^{\omega_4} \to {(V^{\omega_4})}^\vee \otimes V^{\omega_1}$,
  which by adjunction can be rewritten as a $\rG$-equivariant map
  \begin{equation}\label{eq:key-map-in-lemma-on-zero-loci}
    \psi \colon V^{\omega_4} \otimes V^{\omega_4} \to V^{\omega_1}
  \end{equation}
  For any non-zero $h \in H \subset V^{\omega_4}$ we can consider the
  specialisations
  \begin{equation}\label{eq:key-map-in-lemma-on-zero-loci-specialization}
    \begin{aligned}
      & \psi_h \colon V^{\omega_4} \longrightarrow V^{\omega_1} \\
      & \qquad v \mapsto \psi(h \otimes v)
    \end{aligned}
  \end{equation}
  which are nothing else but the maps \eqref{eq:lemma-zero-locus-proof}.

  Combining part~(i) of this lemma with Lemma \ref{lemma:restrictions-of-representations}(4)
  we see that there exist $s_0 \in H$ such that $\ker \psi_{s_0} = V^{\baromega_1}_{\bar{\rG}} \oplus \Bbbk^2$. Since \eqref{eq:key-map-in-lemma-on-zero-loci} is $\rG$-equivariant,
  we have $g (\ker \psi_{s_0}) = \ker \psi_{g \cdot s_0}$ for any $g \in \rG$.
  The subspace $H \subset V^{\baromega_4}$ is preserved by $N_{\rG}(\bar{\rG})$
  and using Lemma \ref{lemma:restrictions-of-representations}(2) we see that
  on the projective line $\bP(H)$ there exists three points $[h] \in \bP(H)$ such
  that for these points $\ker \varphi_h = V^{\baromega_i}_{\bar{\rG}} \oplus \Bbbk^2$
  with distinct $i \in \{1,3,4\}$.

  Finally, since \eqref{eq:key-map-in-lemma-on-zero-loci-specialization} depends
  linearly on $h$, we conclude that these three points are the only points
  on $\bP(H)$ with $\ker \varphi_h = V^{\baromega_i}_{\bar{\rG}} \oplus \Bbbk^2$
  with $i \in \{1,3,4\}$ and for all the other points we have $\ker \varphi_h = \Bbbk^2$.
  Indeed, since for fixed $h$ the map \eqref{eq:key-map-in-lemma-on-zero-loci-specialization}
  is $\bar{\rG}$-equivaraint, from Lemma \ref{lemma:restrictions-of-representations}(4)
  we see that \eqref{eq:key-map-in-lemma-on-zero-loci-specialization} is completely
  determined by the restrictions $f_i(h) \colon V^{\baromega_i}_{\bar{\rG}} \to V^{\baromega_i}_{\bar{\rG}}$, which are linear in $h$. Since each $f_i(h)$ is linear
  in $h$, it either vanishes at one point $[h] \in \bP(H)$ or everywhere. As we have
  already exhibited three distinct points in $\bP(H)$, where only one of the $f_i(h)$
  vanishes, the latter option is not possible.

  We have shown that outside of the three points on $\bP(H)$ constructed above we
  have the desired
  \begin{equation*}
    H^0(\rG/\rP, I_{Z(s)}(1))^\perp = V^{\baromega_2}_{\bar{\rG}}.
  \end{equation*}
  Finally, applying Lemma \ref{lemma:B4-and-D4-as-linear-sections} we get the claim.

\end{proof}

\subsection{Auxiliary complexes}
\label{subsection:auxiliary-complexes}

\begin{lemma}
  \label{lemma:monads}
  \
  \begin{enumerate}
    \item There is a $\rG$-equivariant monad
    \begin{equation}\label{eq:monad-omega_4}
      \cU^{\omega_4}(-1) \to V^{\omega_4} \otimes \cO \to \cU^{\omega_4},
    \end{equation}
    whose middle cohomology is isomorphic to $\cU^{\omega_3}(-1)$.

    \medskip

    \item There is a $\rG$-equivariant monad
    \begin{equation}\label{eq:monad-omega_4-tensor-omega_4}
      \cU^{2\omega_4}(-1) \oplus \cU^{\omega_3}(-1) \oplus \cO \to V^{\omega_4} \otimes \cU^{\omega_4} \to \cU^{2\omega_4} \oplus \cU^{\omega_3} \oplus \cO(1),
    \end{equation}
    whose middle cohomology is isomorphic to $\cU^{\omega_3 + \omega_4}(-1) \oplus \cU^{\omega_2}(-1) \oplus \cU^{\omega_4}$.

    \medskip

    \item There is a $\rG$-equivariant monad
    \begin{equation}\label{eq:monad-omega_4-tensor-omega_2}
      \cU^{\omega_2 + \omega_4}(-1) \oplus \cU^{\omega_3} \to V^{\omega_4} \otimes \cU^{\omega_2} \to \cU^{\omega_2 + \omega_4} \oplus \cU^{\omega_3}(1),
    \end{equation}
    whose middle cohomology is isomorphic to $\cU^{\omega_2 + \omega_3}(-1) \oplus \cU^{\omega_3 + \omega_4} \oplus \cU^{\omega_4}(1)$.

    \medskip

    \item There is a $\rG$-equivariant monad
    \begin{equation}\label{eq:monad-omega_4-tensor-omega_3}
      \cU^{\omega_3 + \omega_4}(-1) \oplus \cU^{\omega_2}(-1) \oplus \cU^{\omega_4} \to V^{\omega_4} \otimes \cU^{\omega_3} \to \cU^{\omega_3 + \omega_4} \oplus \cU^{\omega_2} \oplus \cU^{\omega_4}(1),
    \end{equation}
    whose middle cohomology is isomorphic to $\cU^{2\omega_3}(-1) \oplus \cU^{\omega_2 + \omega_4}(-1) \oplus \cU^{2\omega_4} \oplus \cU^{\omega_3} \oplus \cO(1)$.

    \medskip

    \item There is a $\rG$-equivariant complex
    \begin{equation}\label{eq:monad-2omega_4}
      \cU^{2\omega_4}(-2) \to V^{\omega_4} \otimes \cU^{\omega_4}(-1) \to
      \left( \Lambda^2 V^{\omega_4} \otimes \cO \right) \oplus \left( \cU^{2\omega_4}(-1) \oplus \cU^{\omega_3}(-1) \oplus \cO \right) \to
      V^{\omega_4} \otimes \cU^{\omega_4} \to \cU^{2\omega_4}
    \end{equation}
    whose only (middle) cohomology is isomorphic to $\cU^{2\omega_4}(-1) \oplus \cU^{\omega_2 + \omega_4}(-2)$.

    \medskip

    \item There is a $\rG$-equivariant complex
    \begin{multline}\label{eq:monad-3omega_4}
      \cU^{3\omega_4}(-3) \to V^{\omega_4} \otimes \cU^{2\omega_4}(-2)
      \to \left(\cU^{3\omega_4}(-2) \oplus \cU^{\omega_3 + \omega_4}(-2) \oplus \cU^{\omega_4}(-1) \right) \oplus \left(\Lambda^2 V^{\omega_4} \otimes \cU^{\omega_4}(-1)\right) \to \\
      \left( \Lambda^3 V^{\omega_4} \otimes \cO \right) \oplus \left(V^{\omega_4} \otimes \left( \cU^{2\omega_4}(-1) \oplus \cU^{\omega_3}(-1) \oplus \cO \right) \right) \to \\
      \left(\cU^{3\omega_4}(-1) \oplus \cU^{\omega_3 + \omega_4}(-1) \oplus \cU^{\omega_4} \right) \oplus \left(\Lambda^2 V^{\omega_4} \otimes \cU^{\omega_4}\right) \to
      V^{\omega_4} \otimes \cU^{2\omega_4} \to
      \cU^{3\omega_4}
    \end{multline}
    whose only (middle) cohomology is isomorphic to
    \begin{equation*}
      \cU^{2\omega_2}(-3) \oplus \cU^{\omega_3 + 2\omega_4}(-2) \oplus \cU^{\omega_2 + \omega_4}(-2) \oplus \cU^{2\omega_4}(-1).
    \end{equation*}

    \item There is a $\rG$-equivariant complex
    \begin{multline}\label{eq:monad-2omega_4-tensor-omega_4}
      \cU^{3\omega_4}(-2) \oplus \cU^{\omega_3 + \omega_4}(-2) \oplus \cU^{\omega_4}(-1) \to V^{\omega_4} \otimes \left( \cU^{2\omega_4}(-1) \oplus \cU^{\omega_3}(-1) \oplus \cO \right) \to \\
      \to \left( \Lambda^2 V^{\omega_4} \otimes \cU^{\omega_4} \right) \oplus
      \left( \cU^{3\omega_4}(-1) \oplus \cU^{\omega_3 + \omega_4}(-1) \oplus \cU^{\omega_4} \right) \oplus
      \left( \cU^{\omega_3 + \omega_4}(-1) \oplus \cU^{\omega_2}(-1) \oplus \cU^{\omega_4} \right) \oplus \cU^{\omega_4} \to \\
      \to V^{\omega_4} \otimes \left( \cU^{2\omega_4} \oplus \cU^{\omega_3} \oplus \cO(1) \right) \to
      \cU^{3\omega_4} \oplus \cU^{\omega_3 + \omega_4} \oplus \cU^{\omega_4}(1)
    \end{multline}
    whose only (middle) cohomology is isomorphic to
    \begin{multline*}
      \left( \cU^{3\omega_4}(-1) \oplus \cU^{\omega_3 + \omega_4}(-1) \oplus \cU^{\omega_4} \right) \oplus \\
      \oplus \cU^{\omega_2+2\omega_4}(-2) \oplus \cU^{\omega_2+\omega_3}(-2) \oplus \cU^{\omega_3+\omega_4}(-1) \oplus \cU^{\omega_2}(-1)
    \end{multline*}
  \end{enumerate}
\end{lemma}

\begin{proof}
  Let us first consider \eqref{eq:monad-omega_4}. Since the vector bundle $\cU^{\omega_4}$
  is globally generated by $H^0(X, \cU^{\omega_4}) = V^{\omega_4}$,
  the canonical $\rG$-equivariant morphism $V^{\omega_4} \otimes \cO \to \cU^{\omega_4}$
  is surjective. Taking its dual, using Corollary \ref{corollary:duals} and the selfduality of
  the $\rG$-representation $V^{\omega_4}$, we get the $\rG$-equivariant embedding
  $\cU^{\omega_4}(-1) \to V^{\omega_4} \otimes \cO$.
  The composition
  \begin{equation*}
    \cU^{\omega_4}(-1) \to V^{\omega_4} \otimes \cO \to \cU^{\omega_4}
  \end{equation*}
  vanishes, since it is a $\rG$-equivariant morphism between two non-isomorphic irreducible $\rG$-equivariant vector bundles. Thus, we see that \eqref{eq:monad-omega_4} is indeed a monad.

  A direct computation with weights of $V^{\omega_4}$ shows that among all $26$ weights only
  \begin{equation*}
    \omega_4, \omega_4 - \omega_1, \omega_3 - \omega_1, 0,0
  \end{equation*}
  are $\rL$-dominant. Since the rank of the cohomology of the monad has to be $26 - 6 - 6 = 14$, we see that $\cU^{\omega_3}(-1)$ is the only possibility.

  \medskip

  The remaining monads are derived from \eqref{eq:monad-omega_4} by linear algebra operations:
  \begin{itemize}
    \item \eqref{eq:monad-omega_4-tensor-omega_4} is obtained from \eqref{eq:monad-omega_4}
    by tensoring with $\cU^{\omega_4}$;

    \item \eqref{eq:monad-omega_4-tensor-omega_2} is obtained from \eqref{eq:monad-omega_4}
    by tensoring with $\cU^{\omega_2}$;

    \item \eqref{eq:monad-omega_4-tensor-omega_3} is obtained from \eqref{eq:monad-omega_4}
    by tensoring with $\cU^{\omega_3}$;

    \item \eqref{eq:monad-2omega_4} is the exterior square of \eqref{eq:monad-omega_4};

    \item \eqref{eq:monad-3omega_4} is the exterior cube of \eqref{eq:monad-omega_4};

    \item \eqref{eq:monad-2omega_4-tensor-omega_4} is obtained from \eqref{eq:monad-2omega_4}
    by tensoring with $\cU^{\omega_4}$.
  \end{itemize}
  Tensor products, exterior and symmetric powers required for the above computations
  can be found in Lemmas~\ref{lemma:tensor-products-F4-P1} and \ref{lemma:exterior-powers-omega_4}.
\end{proof}

From the above lemma we deduce the following corollary.
\begin{corollary}\label{corollary:containments}
  We have containments
  \begin{align}
    \label{eq:containment-omega_3}
    & \cU^{\omega_3}(m) \in \cD \quad \text{for} \quad m \in [0,6] \\
    \label{eq:containment-omega_2}
    & \cU^{\omega_2}(m) \in \cD \quad \text{for} \quad m \in [-1,5] \\
    \label{eq:containment-2omega_4}
    & \cU^{2\omega_4}(m) \in \cD \quad \text{for} \quad m \in [0,6] \\
    \label{eq:containment-omega_2+omega_4}
    & \cU^{\omega_2 + \omega_4}(m) \in \cD \quad \text{for} \quad m \in [0,4] \\
    \label{eq:containment-omega_3+omega_4}
    & \cU^{\omega_3 + \omega_4}(m) \in \cD \quad \text{for} \quad m \in [0,5] \\
    \label{eq:containment-omega_2+omega_3}
    & \cU^{\omega_2 + \omega_3}(m) \in \cD \quad \text{for} \quad m \in [0,3] \\
    \label{eq:containment-2omega_3}
    & \cU^{2\omega_3}(m) \in \cD \quad \text{for} \quad m \in [0,4] \\
    \label{eq:containment-3omega_4}
    & \cU^{3\omega_4}(m) \in \cD \quad \text{for} \quad m \in [1,4] \\
    \label{eq:containment-omega_3+2omega_4}
    & \cU^{\omega_3 + 2\omega_4}(m) \in \cD \quad \text{for} \quad m \in [1,3] \\
    \label{eq:containment-2omega_2}
    & \cU^{2\omega_2}(1) \in \cD \\
    \label{eq:containment-omega_2+2omega_4}
    & \cU^{\omega_2 + 2\omega_4}(m) \in \cD \quad \text{for} \quad m \in [1,2]
  \end{align}
\end{corollary}

\begin{proof}
  We treat each case separately:
  \begin{itemize}
    \item \eqref{eq:containment-omega_3} follows from \eqref{eq:monad-omega_4} twisted by
    $\cO(t)$ for $1 \leq t \leq 7$ and \eqref{eq:subcategory-D}.

    \item \eqref{eq:containment-omega_2} follows from \eqref{eq:exact-sequence-tangent-bundle},
    \eqref{eq:exact-sequence-extension-tangent-bundle} twisted by $\cO(t)$ for $0 \leq t \leq 6$ and
    \eqref{eq:subcategory-D}.

    \item \eqref{eq:containment-2omega_4} follows from \eqref{eq:monad-U} twisted by $\cO(t)$
    for $1 \leq t \leq 7$ and \eqref{eq:subcategory-D}.

    \item \eqref{eq:containment-omega_2+omega_4} follows from \eqref{eq:monad-2omega_4} twisted by
    $\cO(t)$ for $2 \leq t \leq 6$, \eqref{eq:containment-2omega_4}, \eqref{eq:containment-omega_3}, and \eqref{eq:subcategory-D}.

    \item \eqref{eq:containment-omega_3+omega_4} follows from \eqref{eq:monad-omega_4-tensor-omega_4} twisted by
    $\cO(t)$ for $1 \leq t \leq 6$, \eqref{eq:containment-2omega_4}, \eqref{eq:containment-omega_3}, and \eqref{eq:subcategory-D}.

    \item \eqref{eq:containment-omega_2+omega_3} follows from \eqref{eq:monad-omega_4-tensor-omega_2}
    twisted by $\cO(t)$ for $1 \leq t \leq 4$,
    \eqref{eq:containment-omega_3}, \eqref{eq:containment-omega_2}, and \eqref{eq:containment-omega_2+omega_4}.

    \item \eqref{eq:containment-2omega_3} follows from \eqref{eq:monad-omega_4-tensor-omega_3}
    twisted by $\cO(t)$ for $1 \leq t \leq 5$,
    \eqref{eq:containment-omega_3}, \eqref{eq:containment-omega_2}, \eqref{eq:containment-omega_3+omega_4},
    \eqref{eq:subcategory-D}.

    \item To show \eqref{eq:containment-3omega_4} we proceed as follows. Tensoring \eqref{eq:monad-U} with $\cU^{\omega_4}$ we get a monad
    \begin{equation*}
      \widetilde{T}_X \otimes \cU^{\omega_4}(-1) \to \left( V^{\omega_1} \oplus \bC \right) \otimes \cU^{\omega_4} \to \widetilde{T}_X \otimes \cU^{\omega_4},
    \end{equation*}
    with cohomology
    \begin{equation*}
      \cU^{2\omega_4}(-1) \otimes \cU^{\omega_4} =
      \cU^{3\omega_4}(-1) \oplus \cU^{\omega_3 + \omega_4}(-1) \oplus \cU^{\omega_4}
      ,
    \end{equation*}
    where we applied Lemma \ref{lemma:tensor-products-F4-P1} to compute the above tensor product.
    Note also that by \eqref{eq:exact-sequence-tangent-bundle}--\eqref{eq:exact-sequence-extension-tangent-bundle}
    and Lemma \ref{lemma:tensor-products-F4-P1} we have
    \begin{equation*}
      \ss(\widetilde{T}_X \otimes \cU^{\omega_4}) = \cU^{\omega_4} \oplus \cU^{\omega_4}(1) \oplus \cU^{\omega_2 + \omega_4}(-1) \oplus \cU^{\omega_3}.
    \end{equation*}
    Finally, twisting the above monad by $\cO(t)$ for $2 \leq t \leq 5$ and using
    the known inclusions
    \eqref{eq:containment-omega_3}, \eqref{eq:containment-omega_2+omega_4},
    \eqref{eq:containment-omega_3+omega_4},
    and~\eqref{eq:subcategory-D}, we get \eqref{eq:containment-3omega_4}.

    \item To show \eqref{eq:containment-omega_3+2omega_4} we proceed as follows. Tensoring \eqref{eq:monad-U} with $\cU^{\omega_3}$ we get a monad
    \begin{equation*}
      \widetilde{T}_X \otimes \cU^{\omega_3}(-1) \to \left( V^{\omega_1} \oplus \bC \right) \otimes \cU^{\omega_3} \to \widetilde{T}_X \otimes \cU^{\omega_3},
    \end{equation*}
    with cohomology
    \begin{equation*}
      \cU^{2\omega_4}(-1) \otimes \cU^{\omega_3} =
      \cU^{\omega_3 + 2\omega_4}(-1) \oplus \cU^{\omega_2 + \omega_4}(-1) \oplus \cU^{2\omega_4} \oplus \cU^{\omega_3}
      ,
    \end{equation*}
    where we applied Lemma \ref{lemma:tensor-products-F4-P1} to compute the above tensor product.
    Note also that by \eqref{eq:exact-sequence-tangent-bundle}--\eqref{eq:exact-sequence-extension-tangent-bundle}
    and Lemma \ref{lemma:tensor-products-F4-P1} we have
    \begin{equation*}
      \ss(\widetilde{T}_X \otimes \cU^{\omega_3}) = \cU^{\omega_3} \oplus \cU^{\omega_3}(1) \oplus \cU^{\omega_2 + \omega_3}(-1) \oplus \cU^{\omega_3 + \omega_4} \oplus \cU^{\omega_4}(1).
    \end{equation*}
    Finally, twisting the above monad by $\cO(t)$ for $2 \leq t \leq 4$ and using
    the known inclusions
    \eqref{eq:containment-omega_3}, \eqref{eq:containment-2omega_4}--\eqref{eq:containment-omega_2+omega_3},
    and \eqref{eq:subcategory-D}, we get \eqref{eq:containment-omega_3+2omega_4}.

    \item \eqref{eq:containment-2omega_2} follows from \eqref{eq:monad-3omega_4} twisted by $\cO(4)$,
    \eqref{eq:containment-omega_3}, \eqref{eq:containment-2omega_4}, \eqref{eq:containment-omega_3+omega_4}
    \eqref{eq:containment-3omega_4}, and~\eqref{eq:subcategory-D}.

    \item \eqref{eq:containment-omega_2+2omega_4} follows from \eqref{eq:monad-2omega_4-tensor-omega_4} twisted by $\cO(t)$
    for $3 \leq t \leq 4$,
    \eqref{eq:containment-omega_3}--\eqref{eq:containment-2omega_4},
    \eqref{eq:containment-omega_3+omega_4},
    \eqref{eq:containment-3omega_4},
    and \eqref{eq:subcategory-D}.
  \end{itemize}
\end{proof}

\subsection{Proof of fullness}
\label{subsection:proof-of-fullness}

In this section we finally prove that the collection \eqref{eq:collection-F4-P1} is full
by restricting it to a family of subvarieties of $X$ of the form $\OG(2,8)$ for which
a convenient full exceptional collection is already known by \cite{KS21}.

Let $Z = \bar{\rG}/\bar{\rP} \cong \OG(2,8)$ and recall that by Lemma \ref{lemma:zero-loci}(2) we have a natural embedding
\begin{equation*}
  i_{\varphi} \colon Z \hookrightarrow X,
\end{equation*}
as the zero locus of a regular section $\varphi \in H^0(X, \cU^{\omega_4}) = V^{\omega_4}$. Acting by an element $g \in \rG$ we get a shifted embedding
\begin{equation*}
  i_{\varphi_{g}} \colon gZ \hookrightarrow X
\end{equation*}
as the zero locus of the section $\varphi_{g} = g \varphi$, where $gZ \cong Z$.
Since the action of $\rG$ on $X$ is transitive, we have the following simple fact.

\begin{lemma}
  Varying $g \in \rG$ we can sweep out the whole $X$ by $gZ \cong \bar{\rG}/\bar{\rP}$.
\end{lemma}

As a corollary we now get the following lemma.

\begin{lemma}[{\cite[Lemma 4.5]{Ku08a}}]
  \label{lemma:kuznetsov}
  If for an object $F \in \Db(X)$ the restrictions $i_{\varphi_{g}}^*F$ vanish for all $\varphi_{g} \in H^0(X, \cU^{\omega_4})$ as above, then $F = 0$.
\end{lemma}

Let us now recall a known exceptional collection for $Z = \OG(2,8)$. Consider full triangulated subcategories
\begin{equation*}
  \cA = \langle \cO, \cU^{\baromega_1}, \cU^{\baromega_3}, \cU^{\baromega_4} \rangle \subset \Db(Z)
\end{equation*}
and
\begin{equation*}
  \cR = \langle \cU^{2\baromega_1}(-1), \cU^{2\baromega_3}(-1), \cU^{2\baromega_4}(-1), \widetilde{T}_{Z}(-1) \rangle \subset \Db(Z),
\end{equation*}
where $\widetilde{T}_{Z}$ is the extension of the tangent bundle $T_Z$ by $\cO_Z$ analogous to \eqref{eq:exact-sequence-extension-tangent-bundle}.
Then, by \cite[Remark 3.23]{KS21} there exists an full exceptional collection of the form
\begin{equation*}
  \Db(Z) = \langle \cR, \cA, \cA(1), \cA(2), \cA(3), \cA(4) \rangle.
\end{equation*}
The left mutation of the subcategory $\langle \cA(2), \cA(3), \cA(4) \rangle$ via $\langle \cR, \cA, \cA(1) \rangle$ is given by
the twist by $\cO(-5)$
and we obtain the decomposition
\begin{equation*}
  \Db(Z) = \langle \cA(-3), \cA(-2), \cA(-1), \cR, \cA(1), \cA(2) \rangle.
\end{equation*}
Finally, twisting by $\cO(3)$ we arrive at
\begin{equation}
  \label{eq:collection-D4-P2}
  \Db(Z) = \langle \cA, \cA(1), \cA(2), \cR(3), \cA(3), \cA(4) \rangle.
\end{equation}
This form of the collection will be used in our proof of fullness below.

\medskip

We also need the following lemma that computes restrictions of \eqref{eq:collection-F4-P1} to $Z$.

\begin{lemma}\label{lemma:lifts-from-Y-to-X}
  There exist isomorphisms
  \begin{equation}\label{eq:restriction-U-omega_4}
    i_{\varphi_{g}}^* \cU^{\omega_4} = \barcU^{\baromega_1} \oplus \barcU^{\baromega_3} \oplus \barcU^{\baromega_4}
  \end{equation}
  \begin{equation}\label{eq:restriction-U-2omega_4}
    i_{\varphi_{g}}^* \cU^{2\omega_4} =
    \barcU^{2\baromega_1} \oplus \barcU^{2\baromega_3} \oplus \barcU^{2\baromega_4} \oplus
    \barcU^{\baromega_1 + \baromega_3} \oplus \barcU^{\baromega_1 + \baromega_4} \oplus \barcU^{\baromega_3 + \baromega_4}
  \end{equation}
  \begin{equation}\label{eq:restriction-tangent-bundle}
    i_{\varphi_{g}}^* \widetilde{T}_X = \widetilde{T}_{gZ} \oplus i_{\varphi_{g}}^* \cU^{\omega_4}
  \end{equation}
\end{lemma}

\begin{proof}
  It is enough to prove the claims for $g = 1$, i.e. we are considering the natural embedding
  $i \colon \bar{\rG}/\bar{\rP} \subset \rG/\rP$.

  To prove \eqref{eq:restriction-U-omega_4} we proceed as follows.
  Since $\cU^{\omega_4}$ is irreducible, the unipotent radical of the parabolic $\rP$
  acts trivially on the $\rP$-representation $E$ corresponding to $\cU^{\omega_4}$
  under the equivalence \eqref{eq:equivalence-homogeneous-bundles-P-reps}.
  The $\bar{\rP}$-representation corresponding to $i^* \cU^{\omega_4}$ under
  \eqref{eq:equivalence-homogeneous-bundles-P-reps} is simply the restriction
  $\Res^{\rP}_{\bar{\rP}} (E)$ of the $\rP$-representation $E$ to $\bar{\rP}$,
  and, since the unipotent radical of $\bar{\rP}$ is contained in the unipotent
  radical of $\rP$, the latter also acts trivially on $E$. Therefore, $\Res^{\rP}_{\bar{\rP}} (E)$
  is completely reducible and we obtain the complete reducibility of $i^* \cU^{\omega_4}$,
  i.e. it is a direct sum of irreducible $\bar{\rG}$-equivariant vector bundles.
  Finally, since $\cU^{\omega_4}$ is generated by its global sections
  $H^0(X, \cU^{\omega_4}) = V^{\omega_4}$, the restriction $i^* \cU^{\omega_4}$
  is also globally generated. Using the description of $V^{\omega_4}|_{\bar{\rG}}$
  given in Lemma \ref{lemma:restrictions-of-representations}(4) and the Borel--Bott--Weil
  theorem we get the desired claim.

  \smallskip

  To prove \eqref{eq:restriction-U-2omega_4} we note that according to Lemma \ref{lemma:exterior-powers-omega_4}
  we have $\cU^{2\omega_4} = S^2 \cU^{\omega_4}$. Hence, we have $i^* \cU^{2\omega_4} = S^2 \left( i^*\cU^{\omega_4} \right)$. We leave the rest as an exercise.

  \smallskip

  To prove \eqref{eq:restriction-tangent-bundle} we proceed as follows. First, we claim that all the statements of
  Lemma~\ref{lemma:definition-of-T_X_tilde} hold verbatim for $T_Z$ and $\widetilde{T}_{Z}$ if we replace
  $\cU^{\omega_2}(-1)$ with $\barcU^{\baromega_1 + \baromega_3 + \baromega_4}(-1)$; the proof is completely
  analogous. In particular, we have $\Ext^\bullet(T_{Z}, \cO) = \Bbbk[-1]$ and $\left( \widetilde{T}_{Z} \right)^\vee = \widetilde{T}_{Z}(-1)$.

  From \eqref{eq:collection-D4-P2} and \eqref{eq:restriction-U-omega_4} we have
  \begin{equation*}
    \Ext^\bullet(\widetilde{T}_{Z}, i^* \cU^{\omega_4}) = 0.
  \end{equation*}
  Dualizing both sides and using the isomorphisms $\left( \widetilde{T}_{Z} \right)^\vee = \widetilde{T}_{Z}(-1)$
  and $\left( i^* \cU^{\omega_4} \right)^\vee = i^* \cU^{\omega_4}(-1)$ we get
  \begin{equation*}
    \Ext^\bullet(i^* \cU^{\omega_4}, \widetilde{T}_{Z}) = 0.
  \end{equation*}
  Since by \eqref{eq:collection-D4-P2}--\eqref{eq:restriction-U-omega_4} we also have $\Ext^\bullet(i^* \cU^{\omega_4}, \cO) = 0$, we obtain $\Ext^\bullet(i^* \cU^{\omega_4}, T_{Z}) = 0$ and the normal bundle exact sequence
  \begin{equation}\label{eq:normal-bundle-ses-1}
    0 \to T_{Z} \to i^* T_X \to i^* \cU^{\omega_4} \to 0
  \end{equation}
  splits.

  Consider the maps
  \begin{equation*}
    \xymatrix{
    \Ext^1(T_X, \cO) \ar[r]^<<<<{\gamma} & \Ext^1(i^*T_X, \cO) \ar[r]^<<<<{\delta} & \Ext^1(T_Z, \cO)
    }
  \end{equation*}
  where $\gamma$ is given by the pullback functor and $\delta$ is given by the morphism
  $T_{Z} \to i^* T_X$. We claim that the composition $\gamma \circ \delta$ is non-zero.
  Indeed, by the Lefschetz principle we can assume $\Bbbk = \bC$. Under the identifications
  \begin{equation*}
    \Ext^1(T_X, \cO) = H^1(X, \Omega_X) = H^2(X, \bC)
    \quad \text{and} \quad
    \Ext^1(T_Z, \cO) = H^1(Z, \Omega_Z) = H^2(Z, \bC)
  \end{equation*}
  the composition $\gamma \circ \delta$ becomes the pullback map on cohomology
  $H^2(X, \bC) \to H^2(Z, \bC)$, which is clearly non-zero.

  Since $\gamma \circ \delta \neq 0$, the pullback
  $i^* \widetilde{T}_X$ is a non-trivial $\bar{\rG}$-equivariant extension of
  $i^* T_X$ by $\cO$. As according to \eqref{eq:normal-bundle-ses-1}
  we have identifications
  $\Ext^\bullet(i^* T_X, \cO) = \Ext^\bullet(T_{Z}, \cO) = \Bbbk[-1]$,
  such an extension is unique.

  Therefore, we can replace the first two terms in \eqref{eq:normal-bundle-ses-1}
  with their unique non-trivial extensions by $\cO$ and get a new short exact sequence
   \begin{equation}\label{eq:normal-bundle-ses-2}
    0 \to \widetilde{T}_{Z} \to i^* \widetilde{T}_X \to i^* \cU^{\omega_4} \to 0.
  \end{equation}
  This finishes the proof.
\end{proof}

\bigskip

We will deduce fullness from the following key fact.
\begin{lemma}\label{lemma:key-lemma}
  Let $\mathbf{\Sigma}$ be the set of vector bundles defined as
  \begin{equation*}
    \mathbf{\Sigma} \coloneqq \{ \cO(t), \cU^{\omega_4}(t), \cU^{2\omega_4}(2) ,\widetilde{T}_X(2) \mid 0 \leq t \leq 4 \}.
  \end{equation*}
  Then for any vector bundle $E \in \mathbf{\Sigma}$ we have inclusions
  \begin{equation}\label{eq:desired-inclusion}
    E \otimes \Lambda^k\cU^{\omega_4} \in \cD \qquad \forall k,
  \end{equation}
  where the subcategory $\cD$ is defined in \eqref{eq:subcategory-D}.
\end{lemma}

\begin{proof}
  We consider four cases:
  \begin{enumerate}
    \item {\bf Case $E \in \{ \cO(t) \mid 0 \leq t \leq 4 \}$.} Using Lemma \ref{lemma:exterior-powers-omega_4} we see that the irreducible summands of $\Lambda^k \cU^{\omega_4}(t)$ for $0 \leq t \leq 4$ are
    \begin{equation*}
      \cO(m) \quad \text{for} \quad m \in [0,7],
    \end{equation*}
    \begin{equation*}
      \cU^{\omega_4}(m) \quad \text{for} \quad m \in [0,6],
    \end{equation*}
    \begin{equation*}
      \cU^{\omega_3}(m) \quad \text{for} \quad m \in [0,5],
    \end{equation*}
    \begin{equation*}
      \cU^{\omega_2}(m) \quad \text{for} \quad m \in [0,4].
    \end{equation*}
    Hence, the desired inclusion \eqref{eq:desired-inclusion} follows from \eqref{eq:containment-omega_3}--\eqref{eq:containment-omega_2}.

    \medskip

    \item {\bf Case $E \in \{\cU^{\omega_4}(t) \mid 0 \leq t \leq 4 \}$.} Using
    Lemmas \ref{lemma:tensor-products-F4-P1} and \ref{lemma:exterior-powers-omega_4} we see that
    the irreducible summands of $\cU^{\omega_4}(t) \otimes \Lambda^k \cU^{\omega_4}$ for $0 \leq t \leq 4$ are
    \begin{equation*}
      \cO(m) \quad \text{for} \quad m \in [1,7],
    \end{equation*}
    \begin{equation*}
      \cU^{\omega_4}(m) \quad \text{for} \quad m \in [0,7],
    \end{equation*}
    \begin{equation*}
      \cU^{\omega_3}(m) \quad \text{for} \quad m \in [0,6],
    \end{equation*}
    \begin{equation*}
      \cU^{\omega_2}(m) \quad \text{for} \quad m \in [0,5],
    \end{equation*}
    \begin{equation*}
      \cU^{2\omega_4}(m) \quad \text{for} \quad m \in [0,6],
    \end{equation*}
    \begin{equation*}
      \cU^{\omega_3 + \omega_4}(m) \quad \text{for} \quad m \in [0,5],
    \end{equation*}
    \begin{equation*}
      \cU^{\omega_2 + \omega_4}(m) \quad \text{for} \quad m \in [0,4].
    \end{equation*}
    Hence, the desired inclusions \eqref{eq:desired-inclusion} hold by Corollary \ref{corollary:containments}.

    \medskip

    \item {\bf Case $E = \cU^{2\omega_4}(2)$.} Using Lemmas \ref{lemma:tensor-products-F4-P1} and \ref{lemma:exterior-powers-omega_4} we see that
    the irreducible summands of $\cU^{2\omega_4}(2) \otimes \Lambda^k \cU^{\omega_4}$ are
    \begin{equation*}
      \cU^{2\omega_4}(m) \quad \text{for} \quad m \in [2,5],
    \end{equation*}
    \begin{equation*}
      \cU^{3\omega_4}(m) \quad \text{for} \quad m \in [2,4],
    \end{equation*}
    \begin{equation*}
      \cU^{\omega_3 + \omega_4}(m) \quad \text{for} \quad m \in [2,4],
    \end{equation*}
    \begin{equation*}
      \cU^{\omega_4}(m) \quad \text{for} \quad m \in [3,5],
    \end{equation*}
    \begin{equation*}
      \cU^{\omega_3 + 2\omega_4}(m) \quad \text{for} \quad m \in [2,3],
    \end{equation*}
    \begin{equation*}
      \cU^{\omega_2 + \omega_4}(m) \quad \text{for} \quad m \in [2,3],
    \end{equation*}
    \begin{equation*}
      \cU^{\omega_3}(m) \quad \text{for} \quad m \in [3,4],
    \end{equation*}
    \begin{equation*}
      \cU^{\omega_2 + 2\omega_4}(2),
    \end{equation*}
    \begin{equation*}
      \cU^{\omega_2}(3).
    \end{equation*}
    Hence, the desired inclusions \eqref{eq:desired-inclusion} hold by Corollary \ref{corollary:containments}.

    \medskip

    \item {\bf Case $E = \widetilde{T}_X(2)$.} To show the desired inclusions \eqref{eq:desired-inclusion}
    it suffices to work with the direct summands of the semisimplification $\ss(\widetilde{T}_X)(2)$.
    By \eqref{eq:exact-sequence-tangent-bundle}--\eqref{eq:exact-sequence-extension-tangent-bundle} we have
    \begin{equation*}
      \ss(\widetilde{T}_X)(2) = \cU^{\omega_2}(1) \oplus \cO(2) \oplus \cO(3).
    \end{equation*}

    Using Lemmas \ref{lemma:tensor-products-F4-P1} and \ref{lemma:exterior-powers-omega_4} we see that
    the irreducible summands of $\ss(\widetilde{T}_X)(2) \otimes \Lambda^k \cU^{\omega_4}$ are
    \begin{equation*}
      \cO(m) \quad \text{for} \quad m \in [2,6],
    \end{equation*}
    \begin{equation*}
      \cU^{\omega_4}(m) \quad \text{for} \quad m \in [2,5],
    \end{equation*}
    \begin{equation*}
      \cU^{\omega_3}(m) \quad \text{for} \quad m \in [2,4],
    \end{equation*}
    \begin{equation*}
      \cU^{\omega_2}(m) \quad \text{for} \quad m \in [1,4],
    \end{equation*}
    \begin{equation*}
      \cU^{\omega_2 + \omega_4}(m) \quad \text{for} \quad m \in [1,3],
    \end{equation*}
    \begin{equation*}
      \cU^{\omega_2 + \omega_3}(m) \quad \text{for} \quad m \in [1,2],
    \end{equation*}
    \begin{equation*}
      \cU^{\omega_3 + \omega_4}(m) \quad \text{for} \quad m \in [2,3],
    \end{equation*}
    \begin{equation*}
      \cU^{2\omega_2}(1),
    \end{equation*}
    \begin{equation*}
      \cU^{2\omega_3}(2),
    \end{equation*}
    \begin{equation*}
      \cU^{2\omega_4}(3).
    \end{equation*}
    Hence, the desired inclusions \eqref{eq:desired-inclusion} hold by Corollary \ref{corollary:containments}.
  \end{enumerate}
\end{proof}

Finally, we are now ready to prove fullness of \eqref{eq:collection-F4-P1}.

\begin{proposition}\label{proposition:fullness}
  If $F \in \Db(X)$ is right orthogonal to all vector bundles $E$ in the exceptional collection \eqref{eq:collection-F4-P1}, i.e. $\Ext^\bullet(E,F) = 0$, then $F = 0$.
\end{proposition}

\begin{proof}
  Recall the definition of the set $\mathbf{\Sigma}$ from Lemma \ref{lemma:key-lemma}.
  By our assumptions we have that for any vector bundle $E \otimes \Lambda^k \cU^{\omega_4}$
  with $E \in \mathbf{\Sigma}$ there are vanishings
  \begin{equation}\label{eq:semiorthogonality-assumption}
    \Ext^{\bullet}(E \otimes \Lambda^k \cU^{\omega_4},F) = H^{\bullet}(X, \left( E \otimes \Lambda^k \cU^{\omega_4} \right)^\vee \otimes F) = 0.
  \end{equation}

  Consider the embedding $i_{\varphi_{g}} \colon gZ \hookrightarrow X$.
  Since the sections $\varphi_{g}$ are regular,
  we have the Koszul resolution
  \begin{equation}
   \label{eq:koszul-complex}
   0 \to \cO(-3) \to \dots \to (\Lambda^2 \cU^{\omega_4})^\vee \to (\cU^{\omega_4})^\vee \to \cO_X \to i_{(\varphi_{g})*} \cO_{gZ} \to 0.
  \end{equation}
  Let us now tensor the Koszul complex \eqref{eq:koszul-complex} by $E^\vee \otimes F$ with $E \in \mathbf{\Sigma}$
  \begin{equation*}
    \dots \to (\Lambda^2 \cU^{\omega_4})^\vee \otimes (E^\vee \otimes F) \to
    (\cU^{\omega_4})^\vee \otimes (E^\vee \otimes F) \to E^\vee \otimes F \to
    i_{(\varphi_{g})*} i_{\varphi_{g}}^*(E^\vee \otimes F) \to 0.
  \end{equation*}
  By \eqref{eq:semiorthogonality-assumption} we have that the cohomology
  groups of $(\Lambda^k \cU^{\omega_4})^\vee \otimes (E^\vee \otimes F)$ vanish. Hence, the cohomology
  groups of $i_{(\varphi_{g})*} i_{\varphi_{g}}^*(E^\vee \otimes F)$ also vanish and we have
  \begin{equation}\label{eq:proof-of-fullness-vanishing}
    0 = H^\bullet(X, i_{(\varphi_{g})*} i_{\varphi_{g}}^*(E^\vee \otimes F)) = H^\bullet(\OG(2,8), i_{\varphi_{g}}^*(E^\vee \otimes F)) = \Ext^\bullet(i_{\varphi_{g}}^* E , i_{\varphi_{g}}^* F)
  \end{equation}
  Hence, $i_{\varphi_{g}}^* F$ is right orthogonal to the objects $i_{\varphi_{g}}^* E$ with $E \in \mathbf{\Sigma}$.

  Finally, by Lemma \ref{lemma:lifts-from-Y-to-X} and \eqref{eq:proof-of-fullness-vanishing}
  we get that $i_{\varphi_{g}}^* F$ is right orthogonal to all objects in the collection
  \eqref{eq:collection-D4-P2}. Hence, we obtain $i_{\varphi_{g}}^* F = 0$. Since this holds
  for any section $\varphi_{g}$, we get $F = 0$ by Lemma \ref{lemma:kuznetsov}.
\end{proof}

\bibliographystyle{plain}
\bibliography{refs}

\begin{bibdiv}
\begin{biblist}

\bib{Akhiezer}{book}{
      author={Akhiezer, Dmitri~N.},
       title={Lie group actions in complex analysis},
      series={Aspects of Mathematics, E27},
   publisher={Friedr. Vieweg \& Sohn, Braunschweig},
        date={1995},
         url={https://doi.org/10.1007/978-3-322-80267-5},
}

\bib{MR4452435}{article}{
      author={Anno, Rina},
      author={Logvinenko, Timothy},
       title={On uniqueness of {$\Bbb{P}$}-twists},
        date={2022},
        ISSN={1073-7928},
     journal={Int. Math. Res. Not. IMRN},
      number={14},
       pages={10533\ndash 10554},
         url={https://doi.org/10.1093/imrn/rnaa371},
      review={\MR{4452435}},
}

\bib{Be}{article}{
      author={Beilinson, A.~A.},
       title={Coherent sheaves on {${\bf P}\sp{n}$}\ and problems in linear
  algebra},
        date={1978},
        ISSN={0374-1990},
     journal={Funktsional. Anal. i Prilozhen.},
      volume={12},
      number={3},
       pages={68\ndash 69},
}

\bib{BS}{article}{
      author={Belmans, P.},
      author={Smirnov, M.},
       title={The {H}ochschild cohomology of generalised {G}rassmannians},
        note={Available at
  \href{https://arxiv.org/abs/1911.09414}{arXiv:1911.09414}},
}

\bib{BKS}{article}{
      author={Belmans, Pieter},
      author={Kuznetsov, Alexander},
      author={Smirnov, Maxim},
       title={Derived categories of the {C}ayley plane and the coadjoint
  {G}rassmannian of type~$\mathrm{F}$},
        date={2021},
     journal={Transformation Groups},
      eprint={https://doi.org/10.1007/s00031-021-09657-w},
         url={https://doi.org/10.1007/s00031-021-09657-w},
        note={See also
  \href{https://arxiv.org/abs/2005.01989}{arXiv:2005.01989}},
}

\bib{Bourbaki}{book}{
      author={Bourbaki, N.},
       title={\'{E}l\'{e}ments de math\'{e}matique. {F}asc. {XXXIV}. {G}roupes
  et alg\`ebres de {L}ie. {C}hapitre {IV}: {G}roupes de {C}oxeter et syst\`emes
  de {T}its. {C}hapitre {V}: {G}roupes engendr\'{e}s par des r\'{e}flexions.
  {C}hapitre {VI}: syst\`emes de racines},
      series={Actualit\'{e}s Scientifiques et Industrielles [Current Scientific
  and Industrial Topics], No. 1337},
   publisher={Hermann, Paris},
        date={1968},
}

\bib{ChPe}{article}{
      author={Chaput, Pierre-Emmanuel.},
      author={Perrin, Nicolas},
       title={On the quantum cohomology of adjoint varieties},
        date={2011},
     journal={Proc. Lond. Math. Soc. (3)},
      volume={103},
      number={2},
       pages={294\ndash 330},
        note={See also
  \href{https://arxiv.org/abs/0904.4824}{arXiv:0904.4824}},
}

\bib{Du}{inproceedings}{
      author={Dubrovin, Boris},
       title={Geometry and analytic theory of {F}robenius manifolds},
        date={1998},
   booktitle={Proceedings of the {I}nternational {C}ongress of
  {M}athematicians, {V}ol. {II} ({B}erlin, 1998)},
       pages={315\ndash 326},
}

\bib{FM}{article}{
      author={Faenzi, Daniele},
      author={Manivel, Laurent},
       title={On the derived category of the {C}ayley plane {II}},
        date={2015},
     journal={Proc. Amer. Math. Soc.},
      volume={143},
      number={3},
       pages={1057\ndash 1074},
        note={See also
  \href{https://arxiv.org/abs/1201.6327}{arXiv:1201.6327}},
}

\bib{Fo19}{article}{
      author={Fonarev, Anton},
       title={Full exceptional collections on {L}agrangian {G}rassmannians},
        date={2020},
     journal={Int. Math. Res. Not.},
        note={Available online ahead of print as
  \href{https://doi.org/10.1093/imrn/rnaa098}{doi:10.1093/imrn/rnaa098}, see
  also \href{https://arxiv.org/abs/1911.08968}{arXiv:1911.08968}},
}

\bib{Gu18}{article}{
      author={Guseva, Lyalya},
       title={On the derived category of {$\mathrm{IGr}(3;8)$}},
        date={2020},
        ISSN={0368-8666},
     journal={Mat. Sb.},
      volume={211},
      number={7},
       pages={24\ndash 59},
        note={See also
  \href{https://arxiv.org/abs/1810.07777}{arXiv:1810.07777}},
}

\bib{Humphreys}{book}{
      author={Humphreys, James~E.},
       title={Linear algebraic groups},
      series={Graduate Texts in Mathematics, No. 21},
   publisher={Springer-Verlag, New York-Heidelberg},
        date={1975},
}

\bib{Jantzen}{book}{
      author={Jantzen, Jens~Carsten},
       title={Representations of algebraic groups},
      series={Pure and Applied Mathematics},
   publisher={Academic Press, Inc., Boston, MA},
        date={1987},
      volume={131},
        ISBN={0-12-380245-8},
}

\bib{Ka}{article}{
      author={Kapranov, M.~M.},
       title={On the derived categories of coherent sheaves on some homogeneous
  spaces},
        date={1988},
     journal={Inventiones Mathematicae},
      volume={92},
      number={3},
       pages={479\ndash 508},
}

\bib{K06}{article}{
      author={Kuznetsov, Alexander},
       title={Hyperplane sections and derived categories},
        date={2006},
     journal={Izv. Ross. Akad. Nauk Ser. Mat.},
      volume={70},
      number={3},
       pages={23\ndash 128},
        note={See also
  \href{https://arxiv.org/abs/math/0503700}{arXiv:math/0503700}},
}

\bib{K07}{article}{
      author={Kuznetsov, Alexander},
       title={Homological projective duality},
        date={2007},
        ISSN={0073-8301},
     journal={Publ. Math. Inst. Hautes \'{E}tudes Sci.},
      number={105},
       pages={157\ndash 220},
         url={https://doi.org/10.1007/s10240-007-0006-8},
        note={See also
  \href{https://arxiv.org/abs/math/0507292}{arXiv:math/0507282}},
}

\bib{Ku08a}{article}{
      author={Kuznetsov, Alexander},
       title={Exceptional collections for {G}rassmannians of isotropic lines},
        date={2008},
     journal={Proceedings of the London Mathematical Society. Third Series},
      volume={97},
      number={1},
       pages={155\ndash 182},
        note={See also
  \href{https://arxiv.org/abs/math/0512013}{arXiv:math/0512013}},
}

\bib{KuPo}{article}{
      author={Kuznetsov, Alexander},
      author={Polishchuk, Alexander},
       title={Exceptional collections on isotropic {G}rassmannians},
        date={2016},
     journal={J. Eur. Math. Soc. (JEMS)},
      volume={18},
      number={3},
       pages={507\ndash 574},
        note={See also
  \href{https://arxiv.org/abs/1110.5607}{arXiv:1110.5607}},
}

\bib{KS21}{article}{
      author={Kuznetsov, Alexander},
      author={Smirnov, Maxim},
       title={Residual categories for (co)adjoint {G}rassmannians in classical
  types},
        date={2021},
     journal={Compositio Mathematica},
      volume={157},
      number={6},
       pages={1172–1206},
        note={See also
  \href{https://arxiv.org/abs/2001.04148}{arXiv:2001.04148}},
}

\bib{Manivel}{article}{
      author={Manivel, Laurent},
       title={On the derived category of the {C}ayley plane},
        date={2011},
     journal={J. Algebra},
      volume={330},
       pages={177\ndash 187},
        note={See also
  \href{https://arxiv.org/abs/0907.2784}{arXiv:0907.2784}},
}

\bib{MR1465519}{incollection}{
      author={Orlov, D.~O.},
       title={Equivalences of derived categories and {$K3$} surfaces},
        date={1997},
      volume={84},
       pages={1361\ndash 1381},
         url={https://doi.org/10.1007/BF02399195},
        note={Algebraic geometry, 7},
      review={\MR{1465519}},
}

\bib{Richardson}{article}{
      author={Richardson, R.~W.},
       title={On orbits of algebraic groups and {L}ie groups},
        date={1982},
        ISSN={0004-9727},
     journal={Bull. Austral. Math. Soc.},
      volume={25},
      number={1},
       pages={1\ndash 28},
         url={https://doi.org/10.1017/S0004972700005013},
}

\bib{sagemath}{manual}{
      author={{The Sage Developers}},
       title={{S}agemath, the {S}age {M}athematics {S}oftware {S}ystem
  ({V}ersion 9.0)},
        date={2020},
        note={{\tt https://www.sagemath.org}},
}

\bib{Timashev}{book}{
      author={Timashev, Dmitry~A.},
       title={Homogeneous spaces and equivariant embeddings},
      series={Encyclopaedia of Mathematical Sciences},
   publisher={Springer, Heidelberg},
        date={2011},
      volume={138},
        ISBN={978-3-642-18398-0},
         url={https://doi.org/10.1007/978-3-642-18399-7},
        note={Invariant Theory and Algebraic Transformation Groups, 8},
}

\bib{VO}{book}{
      author={Vinberg, \`E.~B.},
      author={Onishchik, A.~L.},
       title={{Seminar po gruppam Li i algebraicheskim gruppam}},
     edition={Second},
   publisher={URSS, Moscow},
        date={1995},
}

\end{biblist}
\end{bibdiv}

\end{document}